\documentclass[review,onefignum,onetabnum]{siamart190516}

\usepackage{lipsum}
\usepackage{amsfonts}

\usepackage{graphicx}

\usepackage{epstopdf}
\usepackage{algorithmic}
\ifpdf
  \DeclareGraphicsExtensions{.eps,.pdf,.png,.jpg}
\else
  \DeclareGraphicsExtensions{.eps}
\fi

\usepackage{verbatim}
\usepackage{amsmath}
\usepackage{bbm}
\usepackage{graphicx}
\usepackage{float}
\usepackage{mathrsfs}
\usepackage{tikz-cd}
\usepackage{adjustbox}
\usepackage{multirow}
\usepackage{listings}
\usepackage{xcolor}
\usepackage{enumerate}
\usepackage{thmtools}
\usepackage{thm-restate}
\usepackage{epsfig}
\usepackage{subfigure}
\usepackage[all]{xy}
\usepackage{caption}
\usepackage{longtable}
\usepackage{mathtools}
\usepackage{url}
\usepackage{tcolorbox}
\usepackage{multirow}

\newlength\myindent
\newcommand\bindent{%
  \begingroup
  \setlength{\itemindent}{\myindent}
  \addtolength{\algorithmicindent}{\myindent}
}
\newcommand\eindent{\endgroup}

\newsiamremark{remark}{Remark}
\newsiamremark{example}{Example}
\crefname{hypothesis}{Hypothesis}{Hypotheses}
\newsiamthm{claim}{Claim}
\newtheorem{assumption}{Assumption}

\newcommand{\R}{\mathbb{R}}
\def\conv{\mathrm{conv}}
\DeclareMathOperator*{\argmin}{argmin}
\def\Bern{\mathrm{Bern}}

\definecolor{dgreen}{rgb}{0.00,0.49,0.00}
\definecolor{Brown}{rgb}{0.45,0.0,0.05}

\newcommand{\probaSelec}{P_i^{k,j}}

\headers{Large-scale nonconvex optimization}{J.F. Bonnans, K. Liu, N. Oudjane, L. Pfeiffer, C. Wan}

\title{Large-scale nonconvex optimization: randomization, gap estimation, and numerical resolution\thanks{Submitted to the editors on April 05, 2022.
\funding{J.F.\@ Bonnans was partially supported by the FiME Lab Research Initiative (Institut Europlace de Finance). This article benefited from the support of the FMJH Program PGMO and from the support to this program from EDF.}}}

\author{
J.\@ Frédéric Bonnans\thanks{
Universit{\'e} Paris-Saclay, CNRS, CentraleSup{\'e}lec, Inria, Laboratoire des signaux et syst{\`e}mes, 91190, Gif-sur-Yvette, France
(\email{frederic.bonnans@inria.fr}, \email{laurent.pfeiffer@inria.fr}).
}
\and
Kang Liu\footnotemark[2] \thanks{
Institut Polytechnique de Paris, CNRS, Ecole Polytechnique, CMAP, 91128 Palaiseau, France
(\email{kang.liu@polytechnique.edu}).
}
\and
Nadia Oudjane\thanks{
OSIRIS Department, EDF Lab, Paris-Saclay, and FiME, France
(\email{nadia.oudjane@edf.fr},
\email{cheng.wan.05@polytechnique.org}).
}
\and
Laurent Pfeiffer\footnotemark[2]
\and
Cheng Wan\footnotemark[4]
}

\usepackage{amsopn}

\ifpdf
\hypersetup{
  pdftitle={Large-scale nonconvex optimization: randomization, gap estimation, and numerical resolution},
  pdfauthor={Bonnans, Liu, Oudjane, Pfeiffer, Wan}
}
\fi

\nolinenumbers

\begin{document}

\maketitle

\begin{abstract}
We address a large-scale and nonconvex optimization problem, involving an aggregative term. This term can be interpreted as the sum of the contributions of $N$ agents to some common good, with $N$ large. We investigate a relaxation of this problem, obtained by randomization. The relaxation gap is proved to converge to zeros as $N$ goes to infinity, independently of the dimension of the aggregate. We propose a stochastic method to construct an approximate minimizer of the original problem, given an approximate solution of the randomized problem. McDiarmid's concentration inequality is used to quantify the probability of success of the method. We consider the Frank-Wolfe (FW) algorithm for the resolution of the randomized problem. Each iteration of the algorithm requires to solve a subproblem which can be decomposed into $N$ independent optimization problems. A sublinear convergence rate is obtained for the FW algorithm. In order to handle the memory overflow problem possibly caused by the FW algorithm, we propose a stochastic Frank-Wolfe (SFW) algorithm, which ensures the convergence in both expectation and probability senses. Numerical experiments on a mixed-integer quadratic program illustrate the efficiency of the method.
\end{abstract}



\section{Introduction}

\paragraph{Problem formulation}

This article is devoted to the theoretical analysis and the numerical resolution of the following large-scale, aggregative, and nonconvex optimization problem:
\begin{equation*} \label{pb:aggre} \tag{P}
\inf_{x\in \mathcal{X}}J(x)\coloneqq f(G(x)), \qquad
\text{where: }
\begin{cases}
\begin{array}{rl}
G(x)= & {\displaystyle \frac{1}{N}  \sum_{i=1}^{N}g_i(x_i) } \\[1.5em]
\mathcal{X} = & \prod_{i=1}^N \mathcal{X}_i.
\end{array}
\end{cases}
\end{equation*}
Here, $N$ can be seen as the number of agents and is assumed to be large.
The mappings $g_i \colon \mathcal{X}_i \rightarrow \mathcal{E}$ are given and referred to as the contribution mappings. The space $\mathcal{E}$ is a real Hilbert space.
The main feature of this problem is the aggregative form of the function $G\colon \mathcal{X}\to \mathcal{E}$, which is defined as the average of the $N$ mappings $g_i$.
We will call $G(x)$ the aggregate. Let us emphasize that the dimension $q$ of the aggregate space $\mathcal{E}$ can be arbitrarily large and possibly infinite.
While very few structural assumptions are made on the sets $\mathcal{X}_i$ and the mappings $g_i$, we will assume that $f$ is convex, with a Lipschitz-continuous gradient {and that the image sets $g_i(\mathcal{X}_i)$ are all bounded}. A central idea in this work is that the problem can be well approximated by a convex problem when $N$ is large.

In various examples of interest, the function $f$ has a separable structure as defined below. It turns out that taking into account the separability of $f$, when possible, allows us to refine our theoretical results (more precisely, to reduce some of the constants of interest, see Remark \ref{rm:M}). From now on, we suppose that $\mathcal{E}$ is the Cartesian product of $M$ separable Hilbert spaces denoted $\mathcal{E}_j$, for $j=1,\ldots,M$. We assume that $f$ is additively separable, that is to say, we assume that
\begin{equation*}
    f(y)= \sum_{j=1}^M f_j(y_j), \quad \forall (y_1,\ldots,y_M) \in \prod_{i=1}^M \mathcal{E}_j,
\end{equation*}
where $f_j \colon \mathcal{E}_j \rightarrow \R$.
Note that when $f$ is not separable, one can take $M=1$ and $\mathcal{E}_1= \mathcal{E}$.
We assume that the contribution mappings are of the form
\begin{equation*}
g_i(x_i)=  \big(g_{ij}(x_i) \big)_{j=1,\ldots,M},
\quad
\text{where $g_{ij} \colon \mathcal{X}_i \rightarrow \mathcal{E}_j$.}
\end{equation*}
Hence the criterion $J$ of problem \eqref{pb:aggre} writes
\begin{equation} \label{pb:aggre_revised} 
J(x) = f(G(x)) = \sum_{j=1}^M f_j\Big(\frac{1}{N} \sum_{i=1}^N g_{ij}(x_i) \Big).
\end{equation}
We present and discuss some motivating examples in Section \ref{sec:comments}, arising from social welfare problems, optimal control problems and
supervised learning. 

\paragraph{Related works and methods} 

Let us return to the general problem \eqref{pb:aggre}. Classical Lagrangian relaxation (Chapter XII of \cite{JBHU}) methods can be relevant here because the dual problem is separable in the sense below, thanks to the aggregative form of $G$. To see this, let us reformulate \eqref{pb:aggre} as:  $\inf_{(x,v) \in \mathcal{X} \times \mathcal{E}} f(v)$, subject to the constraint that $v= G(x)$. Its dual problem is:
\begin{equation}\label{eq:lagrange}
	\sup_{\lambda \in \mathcal{E}} \, \big( -f^{*}(\lambda) + \Phi(\lambda) \big),
\end{equation}
where $f^*$ is the Fenchel conjugate function of $f$, and $\Phi(\lambda)$ is defined by
\begin{equation} \label{eq:lagrange2}
\Phi(\lambda) \coloneqq \inf_{x\in \mathcal{X}} \langle \lambda, G(x) \rangle
= \frac{1}{N} \sum_{i=1}^N  \inf_{x_i \in \mathcal{X}_i} \langle \lambda, g_i(x_i) \rangle .
\end{equation}
One sees that $\Phi(\lambda)$ can be evaluated by solving $N$ independent sub-problems, one for each $i$ in $\{1,\ldots,N\}$. 
Solving these sub-problems can be much easier than addressing frontally the original problem with $N$ coupled variables.
This approach has been extensively employed in convex settings \cite{Seguret2020,Pacaud}.
However, the nonconvexity of the problem raises
two major difficulties: the potentially large duality gap  and the reconstruction of a primal solution from the dual optimal solution.

These two difficulties are addressed by Wang in \cite{Wang2017}. She proposed a convex relaxation of the problem, based on a geometrical approach, that allows to obtain an estimate of the duality gap of order $\mathcal{O}(q^2/N^2)$. Her main tool was the Shapley-Folkman lemma \cite{Starr1969},
which allows to show that the image of $G$ is close to a convex set.
This idea was already present in the seminal work of Aubin and Ekeland in \cite{Aubin1976}, dealing with a different setting involving a coupling constraint. {We refer the reader to \cite{kerdreux2022stable} for the most recent improvements dealing with this class of problems.}
We also refer to \cite{Wang2017} for a more exhaustive of mathematical works dedicated to the estimation of the duality gap, where a kind of convexification occurs.
After having solved the dual problem by a cutting plane method and then found an approximate solution to the relaxed primal problem via a projection problem, Wang's method recovers an approximate solution to the original nonconvex problem, by computing a Shapley-Folkman decomposition of the aggregate with a standard linear programming approach.

There exist another important class of methods for large-scale optimization problems which are the block coordinate descent algorithm and its variants \cite{Beck2013, Fercoq2016}. These methods may not be applicable without additional assumptions on the sets $\mathcal{X}_i$ and the maps $g_i$ (in the current framework, the sets $\mathcal{X}_i$ could be discrete). Even if we make additional regularity assumptions, they may be inefficient, in particular because the cost function $J$ is not convex in general.

\paragraph{Contributions and organization of the paper}

We first introduce in Section \ref{sec:rando} a convex relaxation of the original problem \eqref{pb:aggre}. The relaxed problem is obtained by randomization, that is to say, we replace the variables $x_i$ by probability measures $\mu_i$ on $\mathcal{X}_i$. The contribution mappings $g_i(x_i)$ are replaced by $\int_{\mathcal{X}_i} g_i(x_i) d \mu_i(x_i)$; these terms are linear with respect to $\mu_i$. The resulting randomized cost function, denoted $\mathcal{J}$, is convex, and so is the randomized problem. We give a first upper bound of the relaxation gap of order $\mathcal{O}(1/N)$.
The randomized problem has a stochastic interpretation: it amounts to replace the variables $x_i$ by independent random variables $X_i$ of probability distribution $\mu_i$, and to replace $g_i(x_i)$ by the expectation of $g_i(X_i)$. To derive a good candidate (for \eqref{pb:aggre}), given an approximate solution to the randomized problem $\mu= (\mu_1,...,\mu_N)$, we propose to simulate random variables $X_i$ with probability distribution $\mu_i$. We will call this technique the \emph{selection method}.
We give a sharp estimate of the probability of error for the selection method. More precisely, we estimate the probability that $J(X_1,\ldots,X_N) \geq \mathcal{J}(\mu) + \big( \frac{C}{N} + \epsilon \big)$, given $\epsilon > 0$. The proof relies on McDiarmid's inequality, a concentration inequality \cite{Mcdiarmid1989}.

{From a numerical point of view, our main contribution is a method which is parallelizable, which benefits from the convexity of the randomized problem, but avoids the difficulty of the manipulation of probability measures (arising in the formulation of the randomized problem).}
This could be achieved by
combining the Frank-Wolfe (FW) algorithm \cite{Dunn1978,Jaggi2013}, applied to the randomized problem, and the selection method described previously.
The resulting algorithm, called stochastic Frank-Wolfe (SFW) algorithm, is described and analyzed in Section \ref{sec:BFW}. Each iteration of the algorithm requires to solve a subproblem of the form \eqref{eq:lagrange2}, which is decomposable into $N$ subproblems. Resorting to the selection method, we avoid to manipulate explicitely probability measures on the sets $\mathcal{X}_i$, which may otherwise cause memory issues. The SFW method is able to find an $\mathcal{O}(1/N)$-solution to problem \ref{pb:aggre}. In addition, we estimate the probability that the iterate $x_k$ is $\big(\frac{C}{k}+ \epsilon\big)$-optimal, for $k \leq 2N$, where $k$ is the iteration counter. This result relies on concentration inequalities for martingales \cite{Delyon2015} which generalize McDiarmid's inequality.

{Let us note that many articles in the literature are dedicated to stochastic variants of the Frank-Wolfe algorithm. These variants are concerned with the situation where the cost function is in the form of the expectation of a random cost and where its gradient is evaluated by sampling. See for example \cite{ding2018frank,hassani2020stochastic,hazan2016variance,mokhtari2020stochastic,yurtsever2019conditional}, see also \cite{fadili2021inexact,locatello2019stochastic} and the references therein. Let us emphasize that the stochasticity of our algorithm has another origin, namely the selection method.
In all these articles, convergence is established in expectation; to our knowledge, only the article \cite{tang2022high} quantifies the probability of success of some stochastic method based on the Frank-Wolfe algorithm.}

Our last theoretical contribution is a sharp estimate of the relaxation gap, of order $\mathcal{O}(q \wedge N/N^2)$, where $q$ is the (potentially infinite) dimension of the aggregate space $\mathcal{E}$. It is proved in Section \ref{sec:refine}. It relies on a geometrical relaxation of problem \eqref{pb:aggre}, shown to be equivalent to the relaxation by randomization. The relaxation gap is estimated with the help of a measure of nonconvexity for sets (introduced in \cite{Cassels1975}) and with the help of the Shapley-Folkman lemma \cite{Starr1969}. We also give an estimate of the price of decentralization (as defined by Wang in \cite{Wang2017}).
We conclude the section with a detailed comparison of our approach and the one of \cite{Wang2017}.

{Section \ref{sec:comments} is dedicated to examples and discussions on numerical aspects.}
We provide in  section \ref{sec:num} numerical results for a mixed-integer linear-quadratic program.

\subsection{Notations}

\paragraph{On sets} For two sets $\mathcal{A}$ and $\mathcal{B}$  in a normed vector space $\mathcal{X}$, we denote by $d(\mathcal{A}) \colon= \sup_{x,y\in \mathcal{A}} \|x-y\|_{\mathcal{X}}$ the diameter of $\mathcal{A}$,  by $\mathcal{A}+\mathcal{B} = \left\{ x + y \mid x\in \mathcal{A}, y\in \mathcal{B}\right\}$ the Minkowski sum of $A$ and $B$, by $\lambda \mathcal{A} = \{\lambda x \mid x\in \mathcal{A}\}$  the scalar multiplication of $A$ with  $\lambda \in \mathbb{R}$ and by $\conv(\mathcal{A})$ the convex hull of $\mathcal{A}$. Note that $\conv(\mathcal{A} + \mathcal{B})= \conv(\mathcal{A}) + \conv(\mathcal{B})$.

For all $i \in \{ 1,\ldots, N \}$, we denote $\mathcal{X}_{-i}= \big( \prod_{i'=1}^{i-1} \mathcal{X}_{i'} \big) \times \big( \prod_{i'=i+1}^N \mathcal{X}_{i'} \big)$. Given $x \in \mathcal{X}$, we denote $x_{-i}= (x_1,\ldots,x_{i-1},x_{i+1},\ldots, x_N) \in \mathcal{X}_{-i}$. From time to time, we represent $x$ by the pair $(x_i,x_{-i})$.

\paragraph{On functions} Let $\mathcal{H}$ be a {real} Hilbert space. {Let $\langle \cdot, \cdot \rangle_{\mathcal{H}}$ and $\| \cdot \|_{\mathcal{H}}$ denote the corresponding scalar product and norm.} Let $F \colon \mathcal{H} \rightarrow \R \cup \{ + \infty \}$.
The domain of $F$, denoted by $\mathrm{dom}(F)$, is defined by $\mathrm{dom}(F) = \{ x \mid F(x) \neq + \infty\}$. {When $F$ is differentiable, we denote its gradient by $\nabla F$. The gradient is defined as a function from $\mathcal{H}$ to itself. We say that $\nabla F$ is $L$-Lipschitz on a subset $\mathcal{A}$ of $\mathcal{H}$ if for any $x,y\in \mathcal{A}$, we have 
\begin{equation} \label{eq:grad_lip}
\|\nabla F(x) - \nabla F(y)\|_{\mathcal{H}}\leq L \|x-y\|_{\mathcal{H}}.
\end{equation}
}
The subgradient of $F$ at some point $x \in \mathrm{dom}(F)$ is denoted by $\partial F(x)$ and defined by
\begin{equation*}
\partial F(x) = \{ p \in \mathcal{H} \mid F(y ) \geq F(x) + \langle  p,y-x \rangle, \forall \, y \in \mathcal{H}\}.
\end{equation*}
The Fenchel conjugate of $F$ is denoted by $F^{*}\colon H \rightarrow \mathbb{R}$ and defined by
$F^{*}(p) = \sup_{x \in \mathcal{H}} \ \langle p, x\rangle - F(x)$.
 
\paragraph{On measures} Given a set $\Omega$, we denote by $\delta_x$ the Dirac distribution at some point $x \in \Omega$. We denote by $\mathcal{P}_{\delta}(\Omega)$ the set of finitely supported probability distributions, defined by
{
\begin{equation*}
\mathcal{P}_{\delta}(\Omega)
\coloneqq
\Bigg\{
\sum_{k=1}^K \lambda_k \delta_{x_k}
\, \Big| \, K \in \mathbb{N}, \, (\lambda_k)_{k=1}^K \in (\R_+)^K, \, (x_k)_{k=1}^K \in \Omega^K, \, \sum_{k=1}^K \lambda_k = 1
\Bigg\}.
\end{equation*}
}
Let $\mu = \sum_{k=1}^K \lambda_k \delta_{x_k} \in \mathcal{P}_{\delta}(\Omega)$. Given a Hilbert space $\mathcal{H}$ and a mapping $F \colon \Omega \rightarrow \mathcal{H}$, we denote
 \begin{equation*}
 E_{\mu} \big[ F \big]= \sum_{k=1}^K \lambda_k F(x_k), \qquad
 \sigma_{\mu}^2 \big[ F \big]= \sum_{k=1}^K  \lambda_k \big\| F(x_k) -  E_{\mu} \big[ F \big] \big\|_{\mathcal{H}}^2.
 \end{equation*}
In other words, $E_{\mu} \big[ F \big]$ is the integral of $F$ with respect to the measure $\mu$ and $\sigma_{\mu}^2 \big[ F \big]$ is the variance of the probability measure $\sum_{j=1}^J \lambda_j \delta_{F(x_j)}$, in the sense of \cite[Remark 7.5]{Villani2003}.
Finally, the Bernoulli distribution with parameter $\omega\in[0,1]$ is denoted by $\Bern(\omega)$.

\paragraph{On numbers and real-valued random variables} We denote by $m \wedge n$ the minimum of the numbers $m$ and $n$ in $\mathbb{R} \cup \{ + \infty \}$.
Let $X$ be a real-valued random variable. The expectation of $X$ is denoted by $\mathbb{E}[X]$, the variance of $X$ is denoted by $\mathrm{Var}(X)$ and the conditional expectation of $X$ w.r.t.\@ some $\sigma-$algebra $\mathcal{F}$ is denoted by $\mathbb{E}[X\mid \mathcal{F}]$.
Given $\mu \in \mathcal{P}_{\delta}(\Omega)$ and a random variable $X$ in $\Omega$, the notation $X \sim \mu$ indicates that  $\mu$ is the probability distribution of $X$. 

\section{Relaxation by randomization and gap estimation}\label{sec:rando}

In this section we first make a structural assumption on the general problem of interest, problem \eqref{pb:aggre}.
Next we introduce a relaxation of the problem, obtained by randomization. We give an upper bound of the randomization gap in Proposition \ref{prop:gap}. 
Finally we propose a method to recover an approximate solution to \eqref{pb:aggre}, given an approximate solution to the randomized problem.
Its performance is investigated in Theorem \ref{thm:main1}.

\subsection{Assumptions and constants}

{
We recall that $\mathcal{E}$ is the Cartesian product of $M$ separable real Hilbert spaces $\mathcal{E}_j$. We denote by $\langle \cdot, \cdot \rangle_{\mathcal{E}_j}$ the associated scalar products and by $\| \cdot \|_{\mathcal{E}_j}$ the corresponding norms. Let us emphasize that we will not consider any other norm in the spaces $\mathcal{E}_j$. We equip $\mathcal{E}$ with the scalar product $\langle \cdot, \cdot \rangle$, defined by
$\langle (y_1,\ldots,y_M), (y_1',\ldots,y_M') \rangle
= \sum_{j=1}^M \langle y_i, y_i' \rangle_{\mathcal{E}_j}
$
and we denote by $\| \cdot \|$ the corresponding norm.
}

For any $i=1,\ldots,N$ and for any $j=1,\ldots,M$, we denote
\begin{equation*}
S_{ij}\coloneqq \big\{ g_{ij}(x_i) \mid x_i\in\mathcal{X}_i \big\} \qquad \text{and} \qquad S_{j}\coloneqq \frac{1}{N} \sum_{i=1}^N S_{ij}. 
\end{equation*}
The following regularity assumption will be in force all along the article.

\begin{assumption}\label{ass1}  
 For $i=1,2,\ldots,N$ and $j=1,2\ldots,M$:
\begin{enumerate}
\item \label{ass1.1} The range set $S_{ij}$ in $\mathcal{E}_j$ has finite diameter $d_{ij} \coloneqq d(S_{ij})$.
\item \label{ass1.2}The function $f_j $ is $L_j$-Lipschitz on $\conv\,(S_j)$.
\item \label{ass1.3} The function $f_j$ is continuously differentiable on a neighborhood of $\conv\, (S_{j})$, and $\nabla f_j$ is $\tilde{L}_j-$Lipschitz on $\conv\, (S_j)$,  {in the sense of \eqref{eq:grad_lip}}.
\end{enumerate}	
\end{assumption}

We next define two constants $C_0 >0$ and $C_1 > 0$ by
\begin{equation*}
C_0 = \sum_{j=1}^{M} \Big( L_j
\max_{1 \leq i \leq N} \left\{ d_{ij} \right\} \Big),
\quad \text{and} \quad
C_1= \frac{1}{N} \sum_{j=1}^M \Big( \tilde{L}_j \sum_{i=1}^N d_{ij}^2 \Big).
\end{equation*}

\begin{remark}
We will regularly employ notations of the form $O(h(N,q,k))$, where $h$ is an explicit function of $N$, $q$ (the dimension of $\mathcal{E}$), and $k$ (some iteration counter). We use it to express the fact that some variable is bounded by $C\, h(N,q,k)$, where the constant $C$ only depends on $\big(\max_{1 \leq i \leq N} d_{ij} \big)_{j=1,\ldots,M}$ and the Lipschitz moduli $(L_j)_{j=1,\ldots,M}$ and $(\tilde{L}_j)_{j=1,\ldots,M}$. With this convention in mind, we have
\begin{equation*}
C_0 = O(1) \quad \text{and} \quad
C_1 = O(1).
\end{equation*}
\end{remark}

\begin{remark}
Our results can be applied to aggregative problems of the form
\begin{equation*}
\inf_{x \in \mathcal{X}} \
\sum_{j=1}^{M} f_j \Big( \sum_{i=1}^{N} \hat{g}_{ij}(x_i) \Big),
\end{equation*}
i.e.\@ of the same form as in \eqref{pb:aggre}, but without the coefficient $\frac{1}{N}$. Indeed, it suffices to define $g_{ij}= N \hat{g}_{ij}$ to come down to the formulation \eqref{pb:aggre} and to use the fact that $d(g_{ij}(\mathcal{X}_i))= Nd(\hat{g}_{ij}(\mathcal{X}_i))$.
The introduction of the coefficient $\frac{1}{N}$ induces a natural scaling of the problem as $N$ increases. It also enables to us to highlight the convexification of the problem as $N$ becomes large, assuming that the coefficients $d_{ij}$ are uniformly bounded.
\end{remark}

We state in the following lemma a straightforward inequality, exhibiting the role of the constant $C_0$. Note that the role of the constant $C_1$ will be revealed in Lemma \ref{prop:gap}.

\begin{lemma} \label{lemma:easy_ineq}
Let Assumption \ref{ass1} be satisfied.
For all $i \in \{ 1,\ldots, N \}$, for all $x_{-i} \in \mathcal{X}_{-i}$, $x_i$ and $x_i'$ in $\mathcal{X}_i$, it holds:
\begin{equation*}
|J(x_i',x_{-i}) - J(x_i,x_{-i})|
\leq \frac{C_0}{N}.
\end{equation*}
\end{lemma}

\subsection{The randomized problem} \label{subsec:relax}

The \emph{randomized problem} is obtained by replacing each optimization variable $x_i$ by a probability measure $\mu_i \in \mathcal{P}_{\delta}(\mathcal{X}_i)$. The contribution mappings $g_i(x_i)$ are replaced by their integral with respect to $\mu_i$, $E_{\mu_i}\big[ g_i \big]$. Denoting $\mathcal{P}_{\delta}= \prod_{i=1}^{N}\mathcal{P}_{\delta}(\mathcal{X}_i)$, we obtain
\begin{equation*} \label{pb:aggrer}\tag{PR}
\inf_{\mu \in \mathcal{P}_{\delta}}
\mathcal{J}(\mu):= f \Big( \frac{1}{N} \sum_{i=1}^N E_{\mu_i}\big[ g_i \big] \Big)
= \sum_{j=1}^M f_j \Big( \frac{1}{N} \sum_{i=1}^N E_{\mu_i}\big[ g_{ij} \big] \Big).
\end{equation*}
The following equality justifies the denomination of the relaxed problem: given $\mu \in \mathcal{P}_{\delta}$ and given $N$ random variables $X_i$ in $\mathcal{X}_i$ such that $X_i \sim \mu_i$, we have
\begin{equation} \label{eq:relax_random}
\mathcal{J}(\mu)
= f \Big( \frac{1}{N} \sum_{i=1}^N \mathbb{E}\big[ g_i(X_i) \big] \Big).
\end{equation}

\begin{remark}
Working with probability measures with finite support, we do not need to equip the sets $\mathcal{X}_i$ with a topology and to consider regularity assumptions on the mappings $g_i$. Note that the original problem and the randomized one do not necessarily have a solution under the standing assumptions of the article.
\end{remark}

Let $J^{*}$ and $\mathcal{J}^{*}$ denote the values of the primal problem \eqref{pb:aggre} and  the randomized problem \eqref{pb:aggrer} respectively. One is interested in comparing $J^{*}$ and $\mathcal{J}^{*}$. The next lemma gives a direct result for one direction of this comparison.

\begin{lemma} \label{lm: measure problem}
{Let Assumption \ref{ass1} hold true. Then $-\infty < \mathcal{J}^{*}\leq J^{*}$.}
\end{lemma}

\begin{proof}
{By the definitions of $E_{\mu_i}[g_{ij}]$ and $S_j$, we have that $\frac{1}{N}\sum_{i=1}^N E_{\mu_i}[g_{ij}]\in \conv(S_{j})$. Since $f_j$ is Lipschitz-continuous over the bounded set $\conv(S_{j})$, we deduce that $\mathcal{J}^{*}>-\infty$. Let $x \in \mathcal{X}$.} Define $\mu= (\delta_{x_1},\ldots,\delta_{x_N}) \in \mathcal{P}_{\delta}$. Then $\mathcal{J}(\mu)= J(x)$. As a consequence, inequality $\mathcal{J}^{*}\leq J^{*}$ follows.
\end{proof}

The \textit{randomization gap} is then defined as
\begin{equation*}
\text{randomization gap} =J^{*} - \mathcal{J}^{*} \geq 0.
\end{equation*}

Next we prove a first upper bound of the randomization gap, of order $O(\frac{1}{N})$.

\begin{proposition} \label{prop:gap}
Let Assumption \ref{ass1} hold true.
Let $\mu \in \mathcal{P}_{\delta}$ and let $(X_i)_{i=1,\ldots,N}$ denote $N$ independent random variables such that $X_i \sim \mu_i$.
Then,
\begin{equation}\label{eq:gap}
\mathbb{E}[J(X)] - {\mathcal{J}}(\mu)
\leq
\frac{1}{2N^2} \sum_{j=1}^{M} \Big( \tilde{L}_j \sum_{i=1}^{N} \sigma^2_{{\mu_i}} \big[ g_{ij} \big] \Big)
\leq
\frac{C_1}{2N},
\end{equation} 
where $X=(X_1,\ldots,X_N)$.
As a consequence, $J^{*} - \mathcal{J}^{*} \leq \frac{C_1}{2N}$.
\end{proposition}

\begin{proof}
Let us define $Y_j = \frac{1}{N} \big( \sum_{i=1}^N g_{ij}(X_i) \big)$, for $j=1,\ldots,M$. Let us set $Y= (Y_j)_{j=1,\ldots,M}$. We have
\begin{equation*}
\mathbb{E} \big[ J(X) \big]
= \mathbb{E} \big[ f(Y) \big]
\qquad
\text{and}
\qquad
\mathcal{J}(\mu)
= f\big( \mathbb{E}[Y] \big).
\end{equation*}
Since the variables $X_i$ are independent, the random variables $g_{ij}(X_i)$ are also independent (for fixed $j$). It follows that
\begin{equation*}
\mathbb{E}\Big[ \big\| Y_j - \mathbb{E}\big[ Y_j \big] \big\|_{\mathcal{E}_j}^2 \Big]
= \frac{1}{N^2} \sum_{i=1}^N \mathbb{E} \Big[ \big\| g_{ij}(X_i) - \mathbb{E}[g_{ij}(X_i)] \big\|_{\mathcal{E}_j}^2 \Big]
= \frac{1}{N^2} \sum_{i=1}^N \sigma_{\mu_i}^2\big[ g_{ij} \big].
\end{equation*}
By Assumption \ref{ass1}, we have
\begin{equation*}
f(Y)
\leq f\big( \mathbb{E}[Y] \big)
+ \big\langle \nabla f(\mathbb{E}[Y]), Y- \mathbb{E}[Y] \big\rangle_{\mathcal{E}_j}
+ \frac{1}{2}\sum_{j=1}^M \Big( \tilde{L}_j \big\| Y_j - \mathbb{E}\big[ Y_j \big] \big\|_{\mathcal{E}_j}^2 \Big).
\end{equation*}
Taking the expectation of the above inequality and recalling the definition of $C_1$, we deduce \eqref{eq:gap}.
\end{proof}

\begin{remark} \label{rm:M}
{As we explained in the introduction, our analysis covers the case of a non-separable cost $f$ (when $M=1$), however, when $f$ is separable, it is useful to take this property into account. The aim of this remark is to justify this fact. Let us assume (in this remark only) that $f$ is indeed separable, i.e.\@ $M>1$. Let us treat $f$ as a non-separable function.}
It is easy to verify that the mapping $\nabla f$ is Lipschitz continuous with modulus $\big( \max_{j=1,\ldots,M} \tilde{L}_j \big)$; this estimate is tight. If we do not take into account the additive structure of $f$ in the proof of Proposition \ref{prop:gap}, we end up with the following estimate:
\begin{equation*}
\mathbb{E}\big[ J(X) \big] \leq  \mathcal{J}(\mu) + \frac{1}{2N^2} \Big( \max_{j=1,\ldots,M} \tilde{L}_j \Big)
\sum_{i=1}^N \sum_{j=1}^M \sigma_{\mu_i}^2 \big[ g_{ij} \big],
\end{equation*}
which is less precise than inequality \eqref{eq:gap}. The same kind of comment could be made for the constants appearing afterwards in the convergence results of our numerical method.
\end{remark}

We finish this subsection with an equivalent relaxed problem in the situation when the sets $\mathcal{X}_i$ (resp.\@ the contribution functions $g_i$) are identical. We refer to this situation as the \emph{symmetric case}.

\begin{lemma}\label{lm:mean-field}
Suppose that there exists a set $\mathcal{X}$ and a function $g\colon \mathcal{X}\rightarrow \mathcal{E}$ such that $\mathcal{X}_i = \mathcal{X}$ and $g_i = g$, for all $i$. Then,
\begin{equation} \label{eq:mf_relaxation}
    \mathcal{J}^{*} = \inf_{\nu\in \mathcal{P}_{\delta}(\mathcal{X})} \ f\big( E_{\nu} [g] \big).
\end{equation}
\end{lemma}

\begin{proof}
Let $\nu\in \mathcal{P}_{\delta}(\mathcal{X})$. Take $\mu = (\nu,\ldots,\nu) \in \mathcal{P}_{\delta}$. It follows that $ f\left( E_{\nu} [g] \right) = \mathcal{J}(\mu)$. As a consequence, $\inf_{\nu\in \mathcal{P}_{\delta}(\mathcal{X})} f\left( E_{\nu} [g] \right) \leq \inf_{\mu \in \mathcal{P}_{\delta}}
\mathcal{J}(\mu)$. On the other hand, let $\bar{\mu} = (\bar{\mu}_1,\ldots,\bar{\mu}_N) \in \mathcal{P}_{\delta}$. Take $\bar{\nu} = \sum_{i=1}^N \bar{\mu}_i/N \in \mathcal{P}_{\delta}(\mathcal{X})$. Then, we deduce that $ \mathcal{J}(\bar{\mu}) = f\left( E_{\bar{\nu}} [g] \right)$. The conclusion follows.
\end{proof}

The relaxed problem in \eqref{eq:mf_relaxation} has a natural interpretation as a mean field relaxation: instead of considering an optimization problem with $N$ symmetric agents, we consider an arbitrarily large number of agents and optimize their distribution $\nu$.

\color{black}

\subsection{Selection method} \label{subsec:selec_method}

Suppose that a minimizer or an approximate minimizer $\mu$ of the randomized problem \eqref{pb:aggrer} has been obtained.
We address in this subsection the issue of recovering an approximate minimizer of the original problem \eqref{pb:aggre} from $\mu$.

A naive approach would consist in \emph{averaging} the measures $\mu_i$, assuming that the sets $\mathcal{X}_i$ are convex. In such a case, one can define the point $x_i = E_{\mu_i} [ \text{Id}]$.
Another approach, motivated by Proposition \ref{prop:gap}, consists in sampling $\mu$, that is, in simulating $N$ independent random variables $(X_1,\ldots,X_N)$, with distributions $X_i \sim \mu_i$. This can be done without additional structural assumption on the sets $\mathcal{X}_i$, moreover, Proposition \ref{prop:gap} ensures that for any $\varepsilon > 0$,
\begin{equation} \label{eq:proba_sampling}
\mathbb{P} \Big[ J(X_1,\ldots, X_N) < \mathcal{J}(\mu) + \frac{C_1}{2N} + \varepsilon \Big] > 0.
\end{equation}
Of course, one can realize several samplings of $\mu$ to increase the probability of finding a good candidate for the original problem. We will refer to this approach as the \emph{selection method}.

\begin{example}
Consider the following instance of the problem \eqref{pb:aggre}, where $N$ is a large even number:
\begin{equation}\label{exe2}
\left\{\begin{array}{l}
\operatorname{minimize}\left\{	J(x_1,x_2,\ldots,x_N) = -\frac{1}{N}\sum_{i=1}^N x_i^2 + \left(\frac{1}{N}\sum_{i=1}^N x_i\right)^2\right\}; \\
\text { subject to } x_{i} \in [-1,1 ] , \quad i=1, \ldots, N.
\end{array}\right.
\end{equation}
It is easy to see that $x^{*}$ is a minimizer of \eqref{exe2} if and only if $x^{*} $ has $N/2$ coordinates equal to $1$ and the others equal to $-1$. In this example, the original and the relaxed problem have the same value, $J^*= \mathcal{J}^*= -1$.
The relaxed problem does not have a unique solution. One of them is $\tilde{\mu}_i= \frac{1}{2} \big( \delta_{-1} + \delta_{1} \big)$. Averaging $\tilde{\mu}$ as suggested above yields $\tilde{x}= (0,\ldots,0)$ and $J(\tilde{x})= 0$. Thus in this example, the averaging method yields a poor candidate, whatever the value of $N$.

On the other hand, the selection method yields good candidates when $N$ is large. Indeed, assume that $\mathbb{P}[X_i = -1] = \mathbb{P}[X_i= 1]= 1/2$. When $N$ is large, by the law of large numbers \cite{A.B.Tsybakov2009}, nearly half of the random variables $X_i$ are equal to $1$ while the others are equal to $-1$, with probability close to 1. Then in such a case $X$ is almost a minimizer of \eqref{exe2}.
\end{example}

The next theorem provides a sharp estimate of the probability in \eqref{eq:proba_sampling} and confirms the interest of the selection method for large values of $N$. It relies on a concentration inequality, \textit{McDiarmid's inequality} \cite{Mcdiarmid1989}, and its variant \cite{Delyon2015} (cf.\@ Corollary \ref{cor:md}) of ``variance type". It is quite intuitive that if the probability measures $\mu_i$ have a small variance (in a sense to be specified), then the selection method will be more efficient. The interest of taking into account the variances of the probability distributions will be revealed in the analysis of the stochastic Frank-Wolfe algorithm in Subsection \ref{subsec:sfw}.
 
\begin{theorem} \label{thm:main1}
Let Assumption \ref{ass1} be satisfied.
Let $\mu \in \mathcal{P}_{\delta}$ and let $X_1,\ldots,X_N$ be $N$ independent random variables such that $X_i \sim \mu_i$.
Let $X= (X_1,\ldots,X_N)$.
Then, for all $\epsilon > 0$,
\begin{equation}\label{eq:prob_md}
\mathbb{P}\left[ J(X) < \mathcal{J}(\mu)+ \frac{C_1}{2N}  + \epsilon \right]  \geq 1- \exp \left( -\frac{2N\epsilon^{2}}{C_0^2}\right).
\end{equation}
Assume further that for all $i=1,\ldots,N$, there exists a constant $v_i$ such that
\begin{equation} \label{eq:v_i}
\sigma^2_{\mu_i} \big[ J(\cdot,x_{-i}) \big] \leq v_{i}^2,
\end{equation}
for all $x_{-i} \in {X}_{-i}$.
Then \eqref{eq:prob_md} can be strengthened as:
\begin{equation}\label{eq:prob_md1}
\mathbb{P}\left [ J(X) < \mathcal{J}(\mu)+ \sum_{j=1}^{M} \sum_{i=1}^{N} \frac{\tilde{L}_j}{2N^2} \sigma^2_{{\mu_i}} \big[ g_{ij} \big]  + \epsilon \right] \geq 1- \exp \left( -\frac{N\epsilon^{2}}{2\left(\sum_{i=1}^N N v^2_i + \frac{ C_0 \epsilon}{3}\right)}\right) .
\end{equation}
\end{theorem}

\begin{proof}Combining Lemma \ref{lemma:easy_ineq} and McDiarmid's inequality \cite{Mcdiarmid1989}, we obtain
\begin{equation*}
\mathbb{P}\Big[ J(X) < \mathbb{E} \big[ J(X) \big]  + \epsilon \Big]  \geq 1- \exp \left( -\frac{2N\epsilon^{2}}{C_0^2}\right).
\end{equation*}
Combining this estimate with the second inequality of Proposition \ref{prop:gap}, we obtain \eqref{eq:prob_md}.

Estimate \eqref{eq:prob_md1} is proved similarly, combining McDiarmid's inequality of ``variance type" proved in Corollary \ref{cor:md} and the first inequality of Proposition \ref{prop:gap}.
\end{proof}

We provide in the next lemma an explicit candidate for \eqref{eq:v_i}.

\begin{lemma}
Inequality \eqref{eq:v_i} is satisfied with
$v_i^2 = \frac{{1}}{N^2} \big( \sum_{j=1}^M L_j^2 \big) \sigma_{\mu_i}^2(g_{i})$.
\end{lemma}

\begin{proof}
We first state a general following property: given a probability measure $\mu$ and two maps $h_1$ and $h_2$ suitably defined, we have the inequality
{
\begin{equation} \label{eq:inequality}
\sigma_{\mu}^2 \big[ h_1 \circ h_2 \big] \leq  L^2 \sigma_{\mu}^2 \big[ h_2 \big],
\end{equation}
}
assuming that $h_1$ is $L$-Lipschitz continuous.
{
Let us prove this property.
For any $x$, we have
\begin{align*}
& \big\| h_1 \circ h_2 (x) - E_{\mu} [h_1 \circ h_2] \big\|^2
=
\big\| h_1 \circ h_2 (x) - h_1(E_{\mu} [ h_2] ) \big\|^2 \\
& \qquad \qquad +
2 \big\langle h_1 \circ h_2 (x) - h_1(E_{\mu} [ h_2]), h_1(E_{\mu} [ h_2] ) - E_{\mu} [h_1 \circ h_2] \big\rangle \\
& \qquad \qquad + \big\| h_1(E_{\mu} [ h_2] ) - E_{\mu} [h_1 \circ h_2] \big\|^2.
\end{align*}
Taking the expectation, we obtain that
\begin{equation*}
\sigma_{\mu}^2 [h_1 \circ h_2 ]
=  E_\mu \Big[ \big\| h_1 \circ h_2 - h_1(E_{\mu} [ h_2] ) \big\|^2 \Big]
- \big\| h_1(E_{\mu} [ h_2] ) - E_{\mu} [h_1 \circ h_2] \big\|^2.
\end{equation*}
Since $h_1$ is $L$-Lipschitz continuous, we have $E_\mu \big[ \| h_1 \circ h_2 - h_1(E_{\mu} [ h_2] ) \|^2 ] \leq L^2 \sigma_\mu [h_2]^2$. Inequality \eqref{eq:inequality} follows immediately.
}
Next, it is easy to verify that the function $f$ is $L$-Lipschitz continuous, with $L= \big( \sum_{j=1}^M L_j^2 \big)^{1/2}$.
  Using \eqref{eq:inequality}, we conclude that
  {
  \begin{equation*}
  \sigma_{\mu_i}^2 \big[ J(\cdot, x_{-i}) \big] 
  \leq L^2 \sigma_{\mu_i}^2 \Big[ \frac{1}{N} g_i(\cdot) + C \Big]
  = \frac{L^2}{N^2} \sigma_{\mu_i}^2 \big[ g_i \big],
  \end{equation*}
 }
where $C= \frac{1}{N} \sum_{i' \neq i} g_{i'}(x_{i'})$ is regarded as a constant. The estimate follows.
\end{proof}

\section{Stochastic Frank-Wolfe algorithm}\label{sec:BFW}

%

\subsection{Assumptions}

We introduce two new assumptions, which will be in force until the end of the article.

\begin{assumption}\label{ass2}
{For all $j= 1,\ldots,M$, the function $f_j \colon \mathcal{E}_j\rightarrow \mathbb{R}$ is convex over $\conv(S_j)$.}
\end{assumption}

Let $\mu^1$ and $\mu^2$ lie in $\mathcal{P}_{\delta}$. Take $\omega \in [0,1]$. Let $\mu= (\mu_1,\ldots,\mu_N)$ be defined, for any $i=1,\ldots,N$, by $\mu_i= (1-\omega) \mu_i^1 + \omega \mu_i^2$. Here, the addition and the multiplication by a scalar are understood as usual in the set of signed measures. In the sequel, we simply denote $\mu= (1-\omega) \mu^1 + \omega \mu^2$. We have $\mu \in \mathcal{P}_{\delta}$; moreover,
$E_{\mu_i} [g_i]
= (1-\omega) E_{\mu_i^1} [g_i]
+ \omega E_{\mu_i^2}[g_i]$, for any $i=1,\ldots,N$.
Then, Assumption \ref{ass2} implies that $\mathcal{J}(\mu) \leq (1-\omega) \mathcal{J}(\mu^1) + \omega \mathcal{J}(\mu^2)$. In words, the randomized problem \eqref{pb:aggrer} is convex.

In this section, we address the numerical resolution of the randomized problem (and the original problem) under Assumption \ref{ass2}.
Let us mention that this convexity assumption is natural for the application problems described in the introduction.
It allows the application of the Frank-Wolfe algorithm (also called conditional gradient algorithm) \cite{Dunn1978}, for which convergence can be established.
The Frank-Wolfe algorithm requires to solve at each iteration a subproblem. Here, the subproblems can be decomposed in $N$ optimization problems, which can be solved in parallel. This property is particularly interesting, since we aim at solving instances of \eqref{pb:aggre} with large values of $N$.
We do not detail here the practical resolution of the subproblems, which can only be investigated case by case. Instead, we make the following assumption.
{
Let us set $\mathcal{A} \coloneqq \left\{  \nabla f (y) \, \mid\, y\in  \conv(G(\mathcal{X}))  \right\} \subset \mathcal{E}$.
\begin{assumption} \label{ass3}
For all $i=1,\ldots,N$, for all $\lambda \in \mathcal{A}$, the problem
\begin{equation} \label{eq:sub_pb_i}
\inf_{x_i \in \mathcal{X}_i} \, \langle \lambda, g_i(x_i) \rangle
\end{equation}
has at least a solution.
For all $i=1,\ldots,N$, we fix a map $\mathbb{S}_i \colon \mathcal{A} \mapsto \mathcal{X}_i$ such that for any $\lambda \in \mathcal{A}$, $\mathbb{S}_i(\lambda)$ is a solution to \eqref{eq:sub_pb_i}.
\end{assumption}
}

The map $\mathbb{S}_i$ can be understood as a best-response function corresponding to agent $i$. The involved cost function is a linear combination of the contribution mappings $g_{ij}$, with $j=1,\ldots,M$. {In problem \eqref{eq:sub_pb_i}, $\lambda$ can be interpreted as a price variable associated with $g_i(x_i)$.}

\begin{remark}
It is easy to find assumptions which ensure the existence of the map $\mathbb{S}_i$. For example, one can assume that $\mathcal{X}_i$ is a compact set in a topological vector space and that $g_i$ is continuous.
Let us emphasize that Assumption \ref{ass3} is essentially an assumption of numerical nature:
$\mathbb{S}_i$ should be understood as the output of an (efficient) numerical procedure for the resolution of \eqref{eq:sub_pb_i}. The algorithms described afterwards largely rely on evaluations of $\mathbb{S}_i$.
\end{remark}

\subsection{Basic Frank-Wolfe algorithm}

We first describe a rather direct application of the Frank-Wolfe algorithm, which is referred to as the basic Frank-Wolfe algorithm.
The starting point of our numerical approach is the following lemma, the proof of which is straightforward.

\begin{lemma} \label{lm:subpb}
Let {$\lambda \in \mathcal{A}$} and let $\bar{\mu}=(\bar{\mu}_1,\ldots, \bar{\mu}_N) \in \mathcal{P}_{\delta}$. Then,
$\bar{\mu}$ is a solution to
\begin{equation} \label{eq:pb_measure}
\inf_{\mu \in \mathcal{P}_{\delta}} \, \Big\langle {\lambda}, \frac{1}{N} \sum_{i=1}^N E_{\mu_i}(g_i) \Big\rangle.
\end{equation}
if and only if for all $i=1,\ldots,N$, $\bar{\mu}_i$ is supported in $\argmin_{x_i \in \mathcal{X}_i} \, \langle {\lambda}, g_i(x_i) \rangle$.
\end{lemma}

The cost function in \eqref{eq:pb_measure} should be regarded as a linearization of $\mathcal{J}$, as needed in the abstract formulation of the Frank-Wolfe algorithm in \cite{Dunn1978}.
An immediate consequence of Lemma \ref{lm:subpb} is that $(\delta_{\mathbb{S}_1({\lambda})},\ldots,\delta_{\mathbb{S}_N({\lambda})})$ is a solution to \eqref{eq:pb_measure}.
The resolution of problem \eqref{eq:pb_measure} is a key step in the numerical procedures developed afterwards; let us emphasize that the maps $\mathbb{S}_i(y)$ can be evaluated independently from each other, i.e.\@ the resolution of \eqref{eq:pb_measure} can be parallelized.


\begin{algorithm}[H]
\caption{Frank-Wolfe Algorithm }
\label{alg1}
\begin{algorithmic}
\STATE{Initialization: $\mu^0 \in \mathcal{P}_\delta$.}
\FOR{$k= 0,1,\ldots{,K}$}
\bindent
\STATE{\textbf{Step 1: Resolution of the subproblems.}}
\STATE{Set $y^k= \frac{1}{N} \sum_{i=1}^N E_{\mu^k_i}\big[ g_i \big]$ and {set $\lambda^k= \nabla f(y^k)$.}}
\FOR{$i=1,\ldots,N$}
\bindent
\STATE{Compute $\bar{x}_i^k= \mathbb{S}_i({\lambda^k})$.}
\eindent
\ENDFOR
\STATE{Set $\bar{\mu}^k= (\delta_{\bar{x}_1^k},\ldots,\delta_{\bar{x}_N^k})$.}
\STATE{\textbf{Step 2: Update.}}
\STATE{
{Set $\omega_k = 2/(k+2)$.}}
\STATE{Set $\mu^{k+1} = (1-\omega_k) \mu^k + \omega_k \bar{\mu}^k$.}
\eindent
\ENDFOR
\end{algorithmic}
\end{algorithm}

The convergence analysis performed afterwards relies on standard arguments (compare our proof with \cite{Jaggi2013}). We introduce the primal gap $\gamma_k$ and the primal-dual gap $\beta_k$, defined by
\begin{equation} \label{eq:def_gamma_beta}
\gamma_k=  \mathcal{J}(\mu^k) - \mathcal{J}^{*},
\quad
\beta_k = \langle \nabla f(y^k), y^k - \bar{y}^k \rangle,
\quad
\text{where:} \quad
\bar{y}^k = \frac{1}{N} \sum_{i=1}^N g_i(\bar{x}_i^k).
\end{equation}
Note that $\beta_k$ can be evaluated numerically. The following lemma shows that $\beta_k$ is an upper bound of the primal gap $\gamma_k$.

\begin{lemma} \label{lem:upper_bound}
For all $k \in \mathbb{N}$, $\gamma_k \leq \beta_k$.
\end{lemma} 
 
\begin{proof}
Let $k \in \mathbb{N}$. Let $\mu \in \mathcal{P}_\delta$ and let $y= \frac{1}{N} \sum_{i=1}^N E_{\mu_i} [g_i]$. By Lemma \ref{lm:subpb}, we have
$\langle \nabla f(y^k), \bar{y}^k \rangle
\leq \langle \nabla f(y^k), y \rangle$. Thus, using the convexity of $f$, we obtain
\begin{equation}
\beta_k
= \langle \nabla f(y^k), y^k - \bar{y}^k \rangle
\geq \langle \nabla f(y^k), y^k - y \rangle
\geq f(y^k) - f(y)
= \mathcal{J}(\mu^k) - \mathcal{J}(\mu).
\end{equation}
Since $\mu$ is arbitrary, we deduce that $\beta_k \geq \mathcal{J}(\mu^k) - \mathcal{J}^* = \gamma_k$.
\end{proof}

We have the following convergence result.

\begin{proposition} \label{prop:convergence1}
{Let Assumptions \ref{ass1}, \ref{ass2}, and \ref{ass3} hold. Then, in Algorithm \ref{alg1}, for any $K \in \mathbb{N}^*$,
\begin{equation*}
\gamma_K \leq \frac{2C_1}{K}.
\end{equation*}}
\end{proposition}

\begin{proof}
As we will see, the result is a consequence of Lemma \ref{lm: convergence}, with $C= \frac{C_1}{2}$ and $u_k= 0$.
By Assumption \ref{ass1},
\begin{equation*}
f(y^{k+1})
\leq f(y^k)
+ \langle \nabla f(y^k), y^{k+1}- y^k \rangle
+ \sum_{j=1}^M \frac{\tilde{L}_j}{2} \| y_j^{k+1} - y_j^k \|^2.
\end{equation*}
We have $y^{k+1} - y^k= \omega_k (\bar{y}^{k}- y^k )$.
Therefore, by definition of $\beta_k$,
\begin{equation} \label{eq:estim_proof_fw}
f(y^{k+1})
\leq
f(y^k) - \omega_k \beta_k + \omega_k^2
\sum_{j=1}^M \frac{\tilde{L}_j}{2} \| \bar{y}_j^{k}- y_j^k \|^2.
\end{equation}
By definition, $\| \bar{y}_j^{k}- y_j^k \|^2
= \frac{1}{N^2} \big\| \sum_{i=1}^N E_{\mu^k_i} \big[ g_{ij}(\bar{x}_i^k)- g_{ij}(\cdot) \big] \big\|^2$,
thus by Cauchy-Schwarz inequality,
\begin{equation*}
\| \bar{y}_j^{k}- y_j^k \|^2
\leq \frac{1}{N} \sum_{i=1}^N \big\| E_{\mu^k_i} \big[ g_{ij}(\bar{x}_i^k)- g_{ij}(\cdot) \big] \big\|^2
\leq \frac{1}{N} \sum_{i=1}^N d_{ij}^2.
\end{equation*}
Combining the above estimate with \eqref{eq:estim_proof_fw} and using the inequality $\gamma_k \leq \beta_k$ proved in Lemma \ref{lem:upper_bound}, we obtain that
$\gamma_{k+1} \leq (1-\omega_k) \gamma_k + \frac{C_1}{2} \omega_k^2$. Thus Lemma \ref{lm: convergence} applies, which concludes the proof.
\end{proof}

{In the following remark, we give an alternative value of $\omega_k$ in Step 2 of Algorithm \ref{alg1}, while preserving the convergence rate from the previous proposition.}
\begin{remark}\label{rem:ls}
For any $k \in \mathbb{N}$, denote $h_k(\omega)= - \omega \beta_k + \frac{C_k}{2} \omega^2$, where the constant $C_k$ is defined by $C_k = \sum_{j=1}^M \tilde{L}_j \| \bar{y}_j^{k}- y_j^k \|^2$.
In view of inequality \eqref{eq:estim_proof_fw}, the result of Proposition \ref{prop:convergence1} remains true if the sequence $(\omega_k)_{k \in \mathbb{N}}$ is chosen such that for any $k \in \mathbb{N}$, $h(\omega_k) \leq h(\bar{\omega}_k)$. The result remains in particular true for
\begin{equation} \label{eq:ls1}
\omega_k
=
\underset{\omega \in [0,1]}{\text{argmin}} \
h(\omega)
= \min \Big( \frac{\beta_k}{C_k}, 1\Big).
\end{equation}
\end{remark}

The above proposition shows the convergence of the Frank-Wolfe algorithm. Yet the algorithm only provides a relaxed solution. In order to get a solution to the original problem, one can use the selection method introduced in Subsection \ref{subsec:selec_method}. A first direct application of Proposition \ref{prop:gap} yields the following.
Let $(X_1,\ldots,X_N)$ be $N$ independent random variables such that $X_i \sim \mu_i^k$, for all $i$. Then,
\begin{equation*}
\mathbb{E} \big[ J(X) \big] \leq J^{*} + \frac{2C_1}{k} + \frac{C_1}{2N}.
\end{equation*}
Therefore, from a theoretical point of view, there is no guaranty of improvements when $k \gg N$ since, then, the error term $\frac{2C_1}{k}$ becomes negligible in comparison with $\frac{C_1}{2N}$.
The following lemma provides a convergence result (in probability) for the combination of the Frank-Wolfe algorithm and the selection method, for a number of iterations $k \leq N$.

\begin{lemma}
Let $(\mu_k)_{k \in \mathbb{N}}$ be the output of Algorithm \ref{alg1}.
Let $k \leq N$. Let $\zeta \in (0,1)$. Let $n \in \mathbb{N}^*$ and let $(X_i^j)_{i=1,\ldots,N}^{j=1,\ldots,n}$ be $Nn$ independent random variables such that $X_i^j \sim \mu_i^k$. Let $X^j= (X_1^j,\ldots,X_N^j)$. Then,
\begin{equation} \label{eq:n_for_proba}
\mathbb{P} \left[ \min_{j=1,\ldots,n} J(X^j)
< \mathcal{J}^{*} + \frac{3C_1}{k} \right] \geq 1- \zeta,
\quad \text{if} \quad
n \geq \frac{2C_0^2}{C_1^2} \, \frac{k^2}{N} \, \ln \Big( \frac{1}{\zeta} \Big).
\end{equation}
\end{lemma}

\begin{proof}
Since $k \leq N$, we have
$\frac{C_1}{2N} \leq \frac{C_1}{2k}$.
Therefore, by Theorem \ref{thm:main1},
\begin{equation*}
\mathbb{P} \left[ \min_{j=1,\ldots,n} J(X^j)
< \mathcal{J}^{*} + \frac{2C_1}{k} + \frac{C_1}{2k} + \epsilon \right] \geq
1- \exp \Big( - \frac{2N \epsilon^2 n}{C_0^2} \Big),
\end{equation*}
for any $\epsilon >0$. Take $\epsilon = \frac{C_1}{2k}$. If $n$ satisfies \eqref{eq:n_for_proba}, then $
\exp \big( - \frac{2N \epsilon^2 n}{C_0^2} \big) \leq \zeta$.
\end{proof}

\subsection{Stochastic Frank-Wolfe algorithm}
\label{subsec:sfw}

At each iteration of Algorithm \ref{alg1}, a new point $\bar{x}_i^k$ is added to the support of each distribution $\mu^k_i$. Therefore, if at iteration $K$, the points $(\bar{x}^k_i)_{k=0,\ldots,K-1}$ are distinct from each other (for each $i$), then $KN$ places are needed to store the iterate $\mu^K$, which can be prohibitive as $K$ becomes large. We propose in this subsection a variant of Algorithm \ref{alg1} which significantly mitigates the risk of memory overflow. We call it the \emph{Stochastic Frank-Wolfe} (SFW) algorithm, it is given in Algorithm \ref{alg1+k} below.

%
%

\begin{algorithm}[H]
\caption{Stochastic Frank-Wolfe Algorithm }
\label{alg1+k}
\begin{algorithmic}
\STATE{Initialization: $x^0 \in \mathcal{X}$}
\FOR{$k= 0,1,2,\ldots{,K}$}
\bindent
\STATE{\textbf{Step 1: Resolution of the subproblems.}}
\STATE{Compute $y^k= \frac{1}{N} \sum_{i=1}^N g_i(x_i^k)$ and {$\lambda^k= \nabla f(y^k)$}.}
\FOR{$i=1,2,\ldots,N$}
\bindent
\STATE{Compute $\bar{x}_i^k = \mathbb{S}_i ({\lambda^k})$.}
\eindent
\ENDFOR
\STATE{\textbf{Step 2: Update.}}
\STATE{Choose $n_k \in \mathbb{N}^*$.
{Set $\omega_k = 2/(k+2)$.} }
\FOR{$j=1,2,\ldots,n_k$}
\bindent
\FOR{$i=1,2,\ldots,N$}
\STATE{Simulate $\probaSelec{} \sim \Bern(\omega_k)$, independently of all previously defined random variables.}
\STATE{Set $\hat{x}_i^{k,j} = (1-\probaSelec{})x_i^k + \probaSelec{} \bar{x}_i^k$.}
\ENDFOR
\STATE{Define $\hat{x}^{k,j}= (\hat{x}^{k,j}_i)_{i=1,\ldots,N}$.}
\eindent
\ENDFOR
\STATE{Find $x^{k+1} \in \argmin \big\{ J(x) \, \big| \, x \in \{ \hat{x}^{k,j},\, j=1,2,\ldots,n_k \} \big\}$.}
\eindent
\ENDFOR
\end{algorithmic}
\end{algorithm}

Starting from an initialization $x^0 \in \mathcal{X}$, Algorithm \ref{alg1+k} generates a sequence $(x^k)_{k \in \mathbb{N}}$ in $\mathcal{X}$. Let us emphasize that there is no probability distribution involved in the practical implementation of Algorithm \ref{alg1+k}. However, for the analysis of the algorithm and for its description, it is convenient to introduce $\mu^k= (\delta_{x_1^k},\ldots,\delta_{x_N^k})$. With this notation at hand, we first observe that $y^k$, as defined in Step 1 of Algorithm \ref{alg1+k}, satisfies $y^k= \frac{1}{N} \sum_{i=1}^N E_{\mu_i^k} [g_i]$.
Thus the Steps 1 of Algorithms \ref{alg1} and \ref{alg1+k} play exactly the same role. Let us focus next on Step 2 of Algorithm \ref{alg1+k} and let us define $\bar{\mu}^k= (\delta_{\bar{x}_i^1},\ldots,\delta_{\bar{x}_N^k})$ and $\hat{\mu}^k= (1-\omega_k) \mu^k + \omega_k \bar{\mu}^k$. In contrast with Algorithm \ref{alg1}, we do not directly use $\hat{\mu}^k$ at the next iteration but instead employ our selection method so that $\hat{\mu}^k$ is reduced to an $N$-uplet of Dirac measures. The application of the selection method is here simple since $\hat{\mu}_i^k= (1-\omega_k) \delta_{x_i^k} + \omega_k \delta_{\bar{x}_i^k}$. Thus, to simulate a random variable with distribution $\hat{\mu}^k_i$, it suffices to simulate a random variable $P$ with Bernoulli distribution $\Bern(\omega_k)$ and to consider $(1-P)x_i^k + P \bar{x}_i^k$.
Using this method, Step 2 consists in simulating $n_k$ random variables $(\hat{x}^{k,j})_{j=1,\ldots,n_k}$ such that their probability distribution is equal to $\hat{\mu}^k$ (to be rigorous, their probability distribution conditionally to $x^k$). 
Finally, Step 2 selects a random variable $\hat{x}^{k,j}$ which minimizes $J$.

It is important to keep in mind that all variables involved in the algorithm ($x^k$, $\bar{x}^k$, $\hat{x}^{k,j}$) and all variables defined above ($\mu^k$, $\bar{\mu}^k$, $\hat{\mu}^k$) are themselves random variables, since they depend on the Bernoulli random variables $\probaSelec{}$.
For the analysis of the algorithm, we need to consider the filtration generated by the Bernoulli random variables. We introduce the set of indices $\mathcal{I}$ defined by
\begin{equation*}
\mathcal{I}=
\Big\{
(k,j,i) \,|\, k \in \mathbb{N}, \, j \in \{ 1,\dots n_k \}, \, i \in \{ 1,\dots, N \}
\Big\} \cup \big\{ (0,0,0) \big\}.
\end{equation*}
We equip the set $\mathcal{I}$ with the lexicographic order: given $(k_1,j_1,i_1)$ and $(k_2,j_2,i_2)$ in $\mathcal{I}$, we write $(k_1,j_1,i_1) < (k_2,j_2,i_2)$ if and only if
\begin{equation*}
[ k_1 < k_2 ]
\quad \text{or} \quad
[ k_1= k_2 \text{ and } j_1 < j_2]
\quad \text{or} \quad
[(k_1,j_1)=(k_2,j_2) \text{ and } i_1 < i_2].
\end{equation*}
We further write $(k_1,j_1,i_1) \leq (k_2,j_2,i_2)$ if and only if $(k_1,j_1,i_1) < (k_2,j_2,i_2)$ or $(k_1,j_1,i_1) = (k_2,j_2,i_2)$.
Note that this order coincides with the simulation order of the random variables $\probaSelec{}$ in the algorithm. The relation $\leq$ defines a total order with minimal element $(0,0,0)$. For any $(k,j,i) \neq (0,0,0)$, we denote by $(k,j,i)-1$ the maximal element of the set $\{ (k',j',i') \in \mathcal{I} \,|\, (k',j',i') < (k,j,i) \}$.
Finally, we consider the filtration $(\mathcal{G})_{(k,j,i) \in \mathcal{I}}$ defined by
\begin{equation*}
\mathcal{G}_{(k,j,i)} = \left\{
\begin{array}{ll}
\text{trivial $\sigma-$algebra}, & \text{if }  (k,j,i)=(0,0,0), \\
\sigma\big( \mathcal{G}_{(k,j,i)-1}, \probaSelec{} \big), & \text{otherwise,}
\end{array}
\right.
\end{equation*}
where $\sigma\big( \mathcal{G}_{(k,j,i)-1}, \probaSelec{} \big)$ denotes the $\sigma$-algebra generated by $\mathcal{G}_{(k,j,i)-1}$ and $\probaSelec{}$. Note that $\hat{x}_i^{k,j}$ is $\mathcal{G}_{(k,j,i)}$-adapted and that $x^k$ and $\bar{x}^k$ are $\mathcal{G}_{(k,1,1)-1}$-adapted.

\begin{theorem} \label{thm:prob}
Let Assumptions \ref{ass1}, \ref{ass2}, and \ref{ass3} hold true. Then, {for all  $K =1, \ldots, 2N$,}
\begin{equation*}
\mathbb{E} [\gamma_K] \leq \frac{4C_1}{K}, \quad \text{where $\gamma_K= J(x^K)- \mathcal{J}^*$}. 
\end{equation*}
Moreover, for all $\epsilon > 0$,
\begin{equation} \label{eq:prob_estim}
\mathbb{P} \Big[
\gamma_K < \frac{4C_1}{K} + \epsilon
\Big] \geq 1 - \exp \left( \frac{- {\epsilon^2} N}{2(v_K + \epsilon m_K/3)} \right),
\end{equation}
where $v_K= \frac{2 C_0^2}{K^2 (K+1)^2} \ \Bigg(
{\displaystyle \sum_{k=1}^{K-1}}
\frac{k(k+1)^2}{n_k} \Bigg)$
and
$m_K= \frac{C_0}{K(K+1)} \
\Big( \,
{\displaystyle \max_{k=1,\ldots,K-1}
}
\frac{(k+1)(k+2)}{n_k}
\Big).
$
Finally, the following estimates quantify the variability of $\gamma_K$:
\begin{equation} \label{eq:var_estim}
\mathrm{Var}\big[ \gamma_K \big]
\leq \frac{16 C_1^2}{K^2} + \frac{v_K}{N}
\quad \text{and} \quad
\mathbb{E} \Big[
\Big( \max \Big( \gamma_K - \frac{4C_1}{K}, 0 \Big) \Big)^2 \,
\Big] \leq \frac{v_K}{N}.
\end{equation}
\end{theorem}

The proof is postponed to Section \ref{subsec:proofThm}.
Let us note that the constants $m_K$ and $v_K$ involved in the theorem depend on the sequence $(n_k)_{k=0,1,\ldots}$ but do not depend on $N$.

\begin{corollary} \label{coro:convergence}
Let $A > 0$. Assume that $n_k \geq \max\big( \frac{A k^2}{N}, 1 \big)$, for any $k$. Then, for all $K=1,\dots,2N$,
\begin{equation*}
\mathbb{P} \Big[
\gamma_K < \frac{4C_1 + C_0}{K}
\Big] \geq 1 - \exp \left( -\frac{A}{12} \right).
\end{equation*}
\end{corollary}

\begin{proof}
Using $k+1 \leq 2k$, we obtain
\begin{align*}
v_K
\leq \ &
\frac{2C_0^2}{K^2(K+1)^2}
\Bigg(
\sum_{k=1}^{K-1} \frac{N k(k+1)^2}{Ak^2}
\Bigg)
\leq
\frac{8N C_0^2}{A K^2(K+1)^2}
\Bigg(
\sum_{k=1}^{K-1} k
\Bigg) \\
= \ &
\frac{4N C_0^2(K-1)K}{A K^2(K+1)^2}
\leq
\frac{4N C_0^2}{AK^2}
\end{align*}
and $m_K \leq
\frac{C_0}{K(K+1)}
\Big(
{\displaystyle \max_{k=1,\ldots,K-1}}
\frac{N(k+1)(k+2)}{Ak^2}
\Big)
\leq
\frac{6N C_0}{A K^2}$.
Applying Theorem \ref{thm:prob} with $\epsilon= \frac{C_0}{K}$, we obtain that
$\mathbb{P} \Big[
\gamma_K < \frac{4C_1 + C_0}{K}
\Big] \geq 1-p$, with
\begin{equation*}
p \leq
\exp \left(
\frac{-(C_0 / K)^2 N}{2 \Big( \frac{4 N C_0^2}{AK^2} + \frac{6N C_0^2}{3AK^3} \Big) }
\right)
\leq
\exp \left(
\frac{-A}{12}
\right),
\end{equation*}
as was to be proved.
\end{proof}

\begin{remark} \label{rm:modif}
A variant of Algorithm \ref{alg1+k} consists in setting $x^{k+1}=x^k$ if $J(\hat{x}^{k,j}) \geq J(x^k)$ for all $j=1,\ldots,n_k$. Theorem \ref{thm:prob} is still satisfied under this modification.
\end{remark}

\subsection{Proof of Theorem \ref{thm:prob} and comments} \label{subsec:proofThm}

\paragraph{Step 1: proof of the convergence in expectation}
We make use of the notations $\mu^k$, $\bar{\mu}^k$, and $\hat{\mu}^k$, introduced right after Algorithm \ref{alg1+k}. We also introduce
$\beta_k = \langle \nabla f(y^k), y^k - \bar{y}^k \rangle$,
where
$\bar{y}^k = \frac{1}{N} \sum_{i=1}^N g_i(\bar{x}_i^k)$.
By construction, we have
\begin{equation*}
J(x^{k+1})=
\min_{j=1,\ldots,n_k} J(\hat{x}^{k,j})
\leq \frac{1}{n_k} \sum_{j=1}^{n_k} J(\hat{x}^{k,j}).
\end{equation*}
Recalling that $\mathcal{J}(\mu^k)= J(x^k)$, we deduce that $\gamma_{k+1}
\leq \gamma_k + a_k + b_k + c_k$,
where
\begin{align*}
a_k= \ & \frac{1}{n_k} \sum_{j=1}^{n_k} \Big( J(\hat{x}^{k,j}) - \mathbb{E} \big[ J(\hat{x}^{k,j}) \,|\, \mathcal{G}_{(k,1,1)-1} \big] \Big), \\
b_k= \ & \frac{1}{n_k} \sum_{j=1}^{n_k} \Big(  \mathbb{E} \big[ J(\hat{x}^{k,j}) \,|\, \mathcal{G}_{(k,1,1)-1} \big] - \mathcal{J}(\hat{\mu}^k) \Big), \\
c_k= \ &  \mathcal{J}(\hat{\mu}^k)- \mathcal{J}(\mu^k)= \mathcal{J}(\hat{\mu}^k)- J(x^k).
\end{align*}
The term $a_k$ does not play a significant role at the moment since its expectation is null. The term $b_k$ must be understood as a relaxation cost, induced by the use of the selection method. The term $c_k$ is estimated exactly as in Proposition \ref{prop:convergence1}: as was seen in its proof, we have
$c_k \leq  - \omega_k \beta_k + \omega_k^2 \frac{C_1}{2}$.
A direct adaptation of Proposition \ref{prop:gap} shows that
\begin{align*}
b_k
\leq \frac{1}{2N^2} \sum_{j=1}^M \sum_{i=1}^N \tilde{L}_j \sigma_{\hat{\mu}_i^k}^2 [g_{ij}] \leq \frac{1}{2N^2} \sum_{j=1}^M \sum_{i=1}^N \tilde{L}_j \omega_k (1-\omega_k) d_{ij}^2
= \omega_k (1-\omega_k) \frac{C_1}{2N}.
\end{align*}
Combining the above estimates, we obtain
\begin{equation} \label{eq:omega_to_be_minimized}
\gamma_{k+1} \leq \gamma_k  + a_k
+ \Big(
- \omega_k \beta_k + \omega_k^2 \frac{C_1}{2}
\Big)
+ \omega_k (1-\omega_k) \frac{C_1}{2N}.
\end{equation}
For the choice $\omega_k= \bar{\omega}_k$, we have $(1-\omega_k)/N= k/ (N(k+2)) \leq \omega_k$, since $k \leq 2N$. It follows that
\begin{equation*}
\omega_k (1-\omega_k) \frac{C_1}{2N} \leq \omega_k^2 \frac{C_1}{2}
\end{equation*}
and finally, since $\gamma_k \leq \beta_k$, we have
$\gamma_{k+1} \leq (1- \omega_k) \gamma_k + \omega_k^2 C_1 + a_k$.
Next by Lemma \ref{lm: convergence},
\begin{equation} \label{eq:separation}
\gamma_K \leq \frac{4C_1}{K} + S_K, \quad
\text{where: }
S_K = \sum_{k=0}^{K-1} \frac{(k+1)(k+2)}{K(K+1)} a_k.
\end{equation}
We have $\mathbb{E}[a_k]= 0$, thus $\mathbb{E}[S_K]=0$ and finally $\mathbb{E}\big[ \gamma_K \big] \leq \frac{4C_1}{K}$.

\paragraph{Step 2: proof of the probability and variance estimates}
We next need to find an estimate of $\mathbb{P}[S_K \geq \epsilon]$.
For this purpose, we need to further decompose the term $a_k$ as a sum of random variables.
A first observation is the following equality:
$\mathbb{E} \big[ J(\hat{x}^{k,j}) \mid \mathcal{G}_{(k,1,1)-1} \big] 
=
\mathbb{E}\big[J(\hat{x}^{k,j})\mid \mathcal{G}_{(k,j,1)-1}\big],
$
which easily follows from Lemma \ref{lemma_stupid}.
As a consequence,
\begin{equation*}
J(\hat{x}^{k,j})- \mathbb{E}\big[J(\hat{x}^{k,j})\mid \mathcal{G}_{(k,1,1)-1}\big]
= \sum_{i=1}^N U_{(k,j,i)},
\end{equation*}
where
\begin{equation*}
U_{(k,j,i)} = \mathbb{E}\big[ J(\hat{x}^{k,j})\mid \mathcal{G}_{(k,j,i)}  \big]- \mathbb{E}\big[ J(\hat{x}^{k,j})\mid \mathcal{G}_{(k,j,i)-1}\big].
\end{equation*}
We obtain the following decomposition of $S_K$:
\begin{equation*}
S_K =
\sum_{k=1}^{K-1}
\sum_{j=1}^{n_k}
\sum_{i=1}^N
\frac{(k+1)(k+2)}{n_kK(K+1)} \, U_{(k,j,i)}.
\end{equation*}
Note that the index $k$ starts at 1. Indeed, {$\omega_0=1$}, thus $\hat{x}^{0,j}= \bar{x}^0$ and then $a_0= 0$.
Let us apply Proposition \ref{prop:exp} to $S_K$.
We have $\mathbb{E}\big[ U_{(k,j,i)} \mid \mathcal{G}_{(k,j,i)-1} \big]= 0$.
Viewing the term $J(\hat{x}^{k,j})$ as a function $F$ of the random variables
$A:= (P_{i'}^{k',j'})_{(k',j',i') < (k,j,i)}$,
$B:= \probaSelec{}$, and
$C:= (P_{i'}^{k',j'})_{(k,j,i) < (k',j',i')}$,
we can apply Lemma \ref{lem:new_bernoulli} to $U_{(k,j,i)}$, with $\delta= C_0/N$ (by Lemma \ref{lemma:easy_ineq}).
This yields
\begin{align*}
U_{(k,j,i)} \leq \frac{C_0}{N} \qquad
\text{and}
\qquad
\mathbb{E} \big [ U_{(k,j,i)}^2 \mid \mathcal{G}_{(k,j,i)-1} \big]
\leq \frac{\omega_k (1-\omega_k) C_0^2}{N^2}.
\end{align*}
Therefore, Proposition \ref{prop:exp} applies to $\mathbb{P}[S_K \geq \epsilon]$, where the constants $m$ and $v$ are given by
\begin{align*}
m= \ & \max_{k=1,\ldots,K-1}
\frac{(k+1)(k+2)}{n_k K(K+1)}
\frac{C_0}{N}
= \frac{m_K}{N}, \\[1em]
v= \ &
\sum_{k=1}^{K-1}
\sum_{j=1}^{n_k}
\sum_{i=1}^N
\Big(\frac{(k+1)(k+2)}{n_k K(K+1)} \Big)^2 \frac{2k C_0^2 }{(k+2)^2 N^2}
= \frac{v_K}{N}.
\end{align*}
This proves estimate \eqref{eq:prob_estim}.
Recalling that $\gamma_K \leq \frac{4C_1}{K} + S_K$ a.s., we obtain
\begin{equation*}
\mathrm{Var} \big[ \gamma_K \big]
\leq
\mathbb{E} \big[
\gamma_K^2
\big]
\leq
\mathbb{E} \Big[ \Big( \frac{4C_1}{K} + S_K \Big)^2 \, \Big]
= \frac{16 C_1^2}{K^2} + \mathbb{E} \big[ S_K^2 \big].
\end{equation*}
Next by Proposition \ref{prop:exp}, $\mathbb{E}\big[ S_K^2 \big] \leq v_K/N$.
The first inequality in \eqref{eq:var_estim} follows.
The second inequality follows from the inequality: $\max \big( \gamma_K- \frac{4C_1}{K}, 0 )^2 \leq S_K^2$.

\begin{remark}
Let us set $h_k(\omega)= - \omega \beta_k + \omega^2 \frac{C_1}{2} + \omega (1-\omega) \frac{C_1}{2N}$.
If for all $k \in \mathbb{N}$, we have {$h_k(\omega_k) \leq h_k(2/(k+2))$}, then the convergence in expectation of Theorem \ref{thm:prob} still holds, i.e.\@ $\mathbb{E} \big[ \gamma_K\big] \leq 4C_1/K$, in view of inequality \eqref{eq:omega_to_be_minimized}. In particular, one can take
\begin{equation} \label{eq:ls2}
\omega_k=
\underset{\omega \in [0,1]}{\text{argmin}} \
h_k(\omega)
=
\max \Bigg(
\min \Bigg(
\frac{\beta_k- C_1/2N}{C_1(1- 1/N)}
,1 \Bigg)
,0
\Bigg).
\end{equation}
\end{remark}

\subsection{A speed-up of the SFW algorithm}

Step 1 of Algorithm \ref{alg1+k} requires to solve $N$ independent subproblems.
It turns out that only a subset of those subproblems need to be solved for the implementation of Step 2. At iteration $k$ consider the following set:
\begin{equation*}
I_k = \bigcup_{j=1,2,\ldots,n_k} \left\{  i \in \{ 1,\ldots,N \} \,\vert \, \probaSelec{}  =1 \right\}.
\end{equation*}
If $i \notin I_k$, then $\hat{x}_i^{k,j}= x_i^k$, in other words, for such an index $i$, it is not necessary to evaluate $\mathbb{S}_i({\lambda^k})$.
A speed-up of the SFW algorithm can therefore be obtained by simulating the Bernoulli random variables before Step 1, next by evaluating $\mathbb{S}_i({\lambda^k})$ only for the indices $i$ in $I_k$, and finally by computing $\hat{x}^{k,j}$ and $x^{k+1}$ as before. The expectation of the number of subproblems to be solved at iteration $k$ is given by
\begin{align*}
\mathbb{E} \big[ \vert I_k \vert \big]
= \ & \sum_{i=1}^N \mathbb{P}\big[ i \in I_k\big]
= N \big( 1 - \mathbb{P} \big[ 1 \notin I_k \big] \big)
= N \Big( 1 - \mathbb{P} \big[ P_1^{k,j} = 0, \ \forall j=1,\ldots,n_k \big] \Big) \\
= \ & N \Big( 1 - \Big( \frac{k}{k+2} \Big)^{n_k} \Big).
\end{align*}
{
Note that this speed-up technique cannot be applied if $\omega_k$ is chosen according to formula \eqref{eq:ls2}. Indeed, this formula requires to evaluate $\beta_k$, which implies that the $N$ subproblems must all be solved.
}
%
%

\subsection{Stopping time strategy} \label{subsec:stopping}

In Algorithm \ref{alg1+k}, the number of samplings $n_k$ is chosen at the beginning of Step 2. We consider here a variant: we generate a sequence of random variables $\hat{x}^{k,j}$ with probability distribution equal to $\hat{\mu}_k$ (conditionally to $\mathcal{G}_{(k,1,1)-1}$); the variables are constructed via Bernoulli variables independent from each other. We define $n_k$ as the first index $j$ such that
\begin{equation} \label{eq:stopping_rule}
J(\hat{x}^{k,j})
\leq \mathcal{J}(\hat{\mu}^k) + \Big( \frac{C_1}{2} + C_0 \Big) \omega_k^2,
\end{equation}
{
or, equivalently,
\begin{equation}
\label{eq:stopping_rule2}
f(\hat{y}^{k,j}) \leq f\big( (1-\omega_k) y^k + \omega_k \bar{y}^k \big) + \Big( \frac{C_1}{2} + C_0 \Big) \omega_k^2,
\end{equation}
where $\bar{y}^k= \frac{1}{N} \sum_{i=1}^N g_i(\bar{x}_i^k)$ and $\hat{y}^{k,j}= \frac{1}{N} \sum_{i=1}^N g_i(\hat{x}_{i}^{k,j})$.
}
The next iterate is defined by $x^{k+1}= \hat{x}^{k,n_k}$.

\begin{lemma}
Let $(x^k)_{k \in \mathbb{N}}$ denote the sequence obtained with the stopping rule \eqref{eq:stopping_rule}. Then
{
\begin{equation*}
J(x^{K+1})- \mathcal{J}^* \leq \frac{4(C_1 + C_0)}{K}, \quad \forall K=1,\ldots 2N,
\quad \text{a.s.}
\end{equation*}}
Moreover,
\begin{equation*}
\mathbb{E}\big[ n_k \big]
\leq \Big( 1 - \exp \Big(
- \frac{4N}{(k+2)^3}
\Big) \Big)^{-2}, \quad \forall k=1,\ldots, {K}.
\end{equation*}
\end{lemma}

\begin{proof}
Let $\hat{x}$ be a random variable with probability distribution equal to $\hat{\mu}^k$, conditionally to $\mathcal{G}_{(k,1,1)-1}$. Then, for all $\epsilon > 0$, estimate \eqref{eq:prob_md1} of 
Theorem \ref{thm:main1} yields:
\begin{equation} \label{eq:prob_estimate_0}
\mathbb{P} \Big[
J(\hat{x}) \geq \mathcal{J}(\hat{\mu}^k) 
+ \frac{C_1}{2N} \omega_k(1-\omega_k) + \epsilon
\, \Big| \,
\mathcal{G}_{(k,1,1)-1} \Big]
\leq
p_{\epsilon}
\end{equation}
where $p_{\epsilon}= \exp\Big(  \frac{-N \epsilon^2}{2( \omega_k(1-\omega_k) C_0^2 + \frac{C_0}{3} \epsilon )} \Big)$.
For $\epsilon= C_0 \omega_k^2$, we have
\begin{equation*}
p_\epsilon
=
\exp\left(
\frac{-NC_0^2 \omega_k^4}{2 \big( \omega_k C_0^2 - \frac{2}{3} \omega_k^2 C_0^2 \big)}
\right)
\leq
p:=
\exp \left(
\frac{-N \omega_k^3}{2}
\right)
= \exp \left(
\frac{-4N}{(k+2)^3}
\right).
\end{equation*}
Recalling that $\frac{C_1}{2N}\omega_k (1-\omega_k) \leq \frac{C_1}{2} \omega_k^2$, we deduce that
\begin{equation*}
\mathbb{P} \Big[
J(\hat{x}) \geq \mathcal{J}(\hat{\mu}^k) 
+ \Big( \frac{C_1}{2} + C_0 \Big) \omega_k^2
\, \Big| \, \mathcal{G}_{(k,1,1)-1} \Big]
\leq p.
\end{equation*}
Now, let us consider a sequence of independent random variables $(\hat{x}^{k,j})_{j=1,\ldots}$ (conditionally to $\mathcal{G}_{(k,1,1)-1}$), with conditional probability distribution $\hat{\mu}^k$.
By estimate \eqref{eq:prob_estimate_0},
\begin{equation*}
\mathbb{P} \big[ n_k = j \big]
\leq \mathbb{P} \Big[ J(\hat{x}^{k,j'}) \geq
\mathcal{J}(\hat{\mu}^k) + \Big( \frac{C_1}{2} + C_0 \Big) \omega_k^2,\
\forall j'
\,\Big|\,
\mathcal{G}_{(k,1,1)-1}
\Big]
\leq p^{j-1}.
\end{equation*}
We finally deduce that
$\mathbb{E} \big[ n_k \big]
\leq \sum_{n=1}^\infty
j p^{j-1} = \frac{1}{(1-p)^2}$,
which proves the second part of the lemma. For the first part of the lemma, it suffices to observe that
\begin{equation*}
J(x^{k+1})
\leq \mathcal{J}(\hat{\mu}^k)
+ \Big( \frac{C_1}{2} + C_0 \Big) \omega_k^2
\leq J(x^k) - \beta_k \omega_k
+ (C_1 + C_0) \omega_k^2,
\end{equation*}
and to conclude with Lemma \ref{lm: convergence}.
\end{proof}

\subsection{Distributed algorithm}

In this subsection we present a privacy-preserving implementation of Algorithm \ref{alg1+k}. The Algorithm  \ref{alg:distributed} is equivalent to Algorithm \ref{alg1+k}; the instructions are distributed over an \textbf{operator}, $N$ \textbf{agents}, a \textbf{simulator}, and an \textbf{aggregator}, who communicate with each other. Roughly speaking, the operator sets up prices that are sent to the agents, which compute independently from each other their best-response. The {aggregator} computes in a confidential fashion the aggregate associated with a given value of $(x_i)_{i=1,\ldots,N}$. The simulator implements the random variables $P_i^{j,k}$ of the Stochastic Frank-Wolfe algorithm.

%

\begin{algorithm}[H]
\caption{Distributed SFW Algorithm}
\label{alg:distributed}
\begin{algorithmic}
\STATE{[\textbf{Agents}] {Initialization:} $ x^0 \in \mathcal{X}$.}
\STATE{\textbf{[Aggregator]} Compute and send $y^0 = \frac{1}{N} \sum_{i=1}^N g_i(x_i^0)$ to \textbf{Operator}.}
\FOR{$k= 0,1,2,\ldots ,K$}
\bindent
\STATE{\textbf{[Operator] } Compute and send $\lambda^k= \nabla f(y^k)$ to the \textbf{Agents}.}
\FOR{$i=1,2,\ldots,N$}
\bindent
\STATE{\textbf{[Agent $i$] } Compute $\bar{x}_i^k \in \mathbb{S}_i (\lambda^k)$.}
\eindent
\ENDFOR
\STATE{\textbf{[Aggregator]} Compute and send $\bar{y}^k = \frac{1}{N} \sum_{i=1}^N g_i(\bar{x}_i^k)$ to \textbf{Operator}.}
\STATE{\textbf{[Operator] } Compute $\beta^k= \langle \lambda^k, y^k - \bar{y}^k \rangle$.}
\STATE{\textbf{[Operator] } Compute, send $\omega_k$ with \eqref{eq:ls2} or with $\omega_k = \frac{2}{k+2}$ to \textbf{Simulator}.}
\STATE{\textbf{[Operator]} Set $j=0$ and send $test =$ true to \textbf{Simulator}.}
\WHILE{$test$}
\bindent
\STATE{[\textbf{Operator}] Increment $j$.}
\FOR{$i=1,2,\ldots,N$}
\bindent
\STATE{\textbf{[Simulator]} Simulate and send $\probaSelec{} \sim \Bern(\omega_k)$ to \textbf{Agent} i.}
\STATE{[\textbf{Agent} $i$] Set $\hat{x}_i^{k,j} = (1- \probaSelec{} )x_i^k + \probaSelec{}  \bar{x}_i^k$. }
\eindent
\ENDFOR
\STATE{\textbf{[Aggregator] } Compute, send $\hat{y}^{k,j} = \frac{1}{N}\sum_{i=1}^N g_i(\hat{x}_i^{k,j})$ to  \textbf{Operator}.}
\STATE{[\textbf{Operator}] Update and send $test$ to \textbf{Simulator}.}
\eindent
\ENDWHILE
\STATE{\textbf{[Operator]} Find $j^{*} \in \underset{j'=1,\ldots,j}{\argmin} \ f(\hat{y}^{k,j'})$.
Set $y^{k+1} = \hat{y}^{k,j^{*}}$.}
\STATE{\textbf{[Operator]} Send $j^{*}$ to the \textbf{Agents}.}
\FOR{$i=1,2,\ldots,N$}
\bindent
\STATE{[\textbf{Agent} $i$] Set $x_i^{k+1} = \hat{x}_i^{k,j^{*}}$.}
\eindent
\ENDFOR
\eindent
\ENDFOR
%
%
%
\end{algorithmic}
\end{algorithm}

More precisely, at the beginning of iterattion $k$ of Algorithm \ref{alg:distributed}, the operator sends a price $\lambda_k$ to the agents, who calculate their best-response. The aggregator sends the corresponding aggregate $\bar{y}_k$ to the operator, who can compute the primal-dual gap $\beta^k$ and can fix the value of the stepsize $\omega_k$.
Next the simulator realizes stochastic simulations, communicated to the agents. Only the aggregate associated with each simulation, $\hat{y}^{k,j}$, is communicated to the operator. The operator decides when to stop the simulation phase through the logical variable $test$. For example, $test$ can be set to true as long as $j< n_k$, for predefined values of $n_k$. The variable test can also be designed so as to implement the stopping rule \eqref{eq:stopping_rule2} of Subsection \ref{subsec:stopping}.
Finally, the operator identifies the number $j^*$ of the simulation that has yielded the best aggregate and communicates it to the agents.

The key point in this algorithm is that the operator never receives information that is specific to a given agent: it only collects aggregates (the variables $\bar{y}^k$, $\hat{y}^{k,j}$, and $y^k$). Similarly, the agents have only access to the prices $\lambda_k$ and to $j^*$. We do not detail here algorithms used by the aggregator to compute the aggregate and refers the reader to \cite{beaude2020privacy}, which investigates  a similar approach for preserving privacy, with an operator that only has access to aggregates (note that the underlying mathematical method is different from ours). It is proposed in that reference to use a cryptographic protocol called \emph{secure multiparty computation} for the non-intrusive computation of aggregates, taken from \cite{shi2016secure} and \cite{atallah2004private}.

\color{black}

\section{Refined gap estimates} \label{sec:refine}

\subsection{Nonconvexity measure and gap estimate}
\label{subsec-refined_gaps}

We give in this subsection a refinement of the randomization gap obtained in Proposition \ref{prop:gap}. Our analysis relies on the concept of nonconvexity measure, introduced in \cite{Cassels1975}.

\begin{definition}
Given a subset $\mathcal{K}$ of $\mathcal{E}$, we call nonconvexity measure of  $\mathcal{K}$ the number $\rho(\mathcal{K})$ defined by
\begin{equation*}
\rho(\mathcal{K})
= \Bigg(
\sup_{y \in \conv(\mathcal{K})} \,
\inf_{\begin{subarray}{c} \mu \in \mathcal{P}_\delta, \\
E_{\mu} [ \mathrm{Id} ]= y \end{subarray}} \sigma_{\mu}\big[ \mathrm{Id} \big]^2
\Bigg)^{1/2},
\end{equation*}
where $\mathrm{Id} \colon \mathcal{E} \rightarrow \mathcal{E}$ denotes the identity mapping.
\end{definition}

The "nonconvexity measure” terminology is motivated by the following: if $\mathcal{K}$ is convex, then obviously $\rho(\mathcal{K})= 0$ and conversely, if $\rho(\mathcal{K})=0$, then $\mathcal{K}$ is dense into $\conv(\mathcal{K})$.
We have the following two properties, easily verified. The map $\rho$ is homogeneous in the following sense: given $a \in \R$, we have $\rho(a \mathcal{K})= |a| \rho(\mathcal{K})$. Moreover $\rho(\mathcal{K}) \leq d(\mathcal{K})$, where $d(\mathcal{K})$ is the diameter of $\mathcal{K}$. Another particularly interesting property for our aggregative problem is the sub-additivity of $\rho(\cdot)^2$: given two subsets $\mathcal{K}_1$ and $\mathcal{K}_2$, we have
$\rho(\mathcal{K}_1 + \mathcal{K}_2)^2
\leq \rho(\mathcal{K}_1)^2
+ \rho(\mathcal{K}_2)^2$, see \cite[Theorem 1]{Cassels1975}. 
We will use an improvement of this inequality in the proof of Theorem \ref{thm:new_gap}, based on the Shapley-Folkman theorem.

The next lemma provides a general relaxation estimate based on the nonconvexity measure of the feasible set. Let us emphasize that the central idea behind this result is the same as the one in the proof of Proposition \ref{prop:gap}. The only difference is the point of view, which is here geometric while it was previously probabilistic.

\begin{lemma} \label{lemma:gap_geo}
Let $\mathcal{K}$ be a subset of $\mathcal{E}$. Let $F$ be a differentiable real-valued function defined on some neighborhood of $\conv(\mathcal{K})$. Assume that $\nabla F$ is $\tilde{L}$-Lipschitz continuous over $\conv(\mathcal{K})$. Then,
\begin{equation*}
\inf_{y \in \mathcal{K}} F(y)
\leq
\Big( \inf_{y \in \conv(\mathcal{K})} F(y)
\Big)
+ \frac{\tilde{L}}{2} \rho(\mathcal{K})^2.
\end{equation*}
\end{lemma}

\begin{proof}
Let $y \in \conv(\mathcal{K})$. Let $\mu \in \mathcal{P}_{\delta}(\mathcal{K})$ be such that $E_{\mu}[\mathrm{Id}]= y$. Then, since $\nabla F$ is $\tilde{L}$-Lipschitz continuous, we have
\begin{equation*}
\inf_{y' \in \mathcal{K}} F(y')
\leq E_{\mu}[F]
\leq F(y) + \frac{\tilde{L}}{2} \sigma_{\mu}^2\big[ \mathrm{Id} \big].
\end{equation*}
Minimizing the right-hand side with respect to $\mu$, we obtain that
\begin{equation*}
\inf_{y' \in \mathcal{K}} F(y')
\leq F(y) + \frac{\tilde{L}}{2} \rho(\mathcal{K})^2.
\end{equation*}
Minimizing the result with respect to $y$ yields the announced estimate.
\end{proof}

Some notations are needed for the application of Lemma \ref{lemma:gap_geo} to \eqref{pb:aggre}. We set
\begin{equation*}
\begin{array}{ll}
\tilde{g}_{ij}(x_i)
= \sqrt{\tilde{L}_j} \, g_{ij}(x_i),
\qquad & \tilde{g}_i(x_i)
=
\big( \tilde{g}_{ij} (x_i) \big)_{j=1,\ldots,M}
\\[1em]
\tilde{f}_j(y_j)= f_j \Big( \tfrac{y_j}{ \sqrt{\tilde{L}_j} } \Big),
& \tilde{f}(y)={\displaystyle \sum_{j=1}^M \tilde{f}_j(y_j).}
\end{array}
\end{equation*}
Obviously,
$J(x)= \tilde{f} \big( \frac{1}{N} \sum_{i=1}^N \tilde{g}_i(x_i) \big)
= \sum_{j=1}^{M}\tilde{f}_j \big( \frac{1}{N}\sum_{i=1}^{N} \tilde{g}_{ij}(x_i) \big)$.
Finally we denote
\begin{equation*}
\mathcal{Y}_i= \tilde{g}_i(\mathcal{X}_i)
\quad
\text{and}
\quad
\mathcal{Y}= \frac{1}{N} \sum_{i=1}^{N} \mathcal{Y}_i.
\end{equation*}

We give next two new formulations of problems \eqref{pb:aggre} and \eqref{pb:aggrer}, revealing the geometric nature of the relaxation technique employed so far.

\begin{lemma} \label{lemma:gap_estimate}
We have
\begin{align*}
J^*= & \ \inf_{y \in \mathcal{Y}} \ \tilde{f} (y), \tag{PG} \\[1em]
\mathcal{J}^*= & \ \inf_{y \in \conv(\mathcal{Y})} \ \tilde{f}(y). \tag{PGR} \label{pb:rel_geo}
\end{align*}
\end{lemma}

\begin{proof}
The first equality is straightforward. For the second one, it suffices to observe that
$\conv(\mathcal{Y})= \frac{1}{N} \sum_{i=1}^N \conv(\mathcal{Y}_i)$
and that
$\conv(\mathcal{Y}_i)= \big\{ E_{\mu_i}[\tilde{g}_i] \mid \mu_i \in \mathcal{P}_{\delta}(\mathcal{X}_i) \big\}$.
\end{proof}

We introduce the following constants:
\begin{equation} \label{eq:D_i}
D_i = \sum_{j=1}^M \tilde{L}_j d^2_{ij}, \qquad
D[k]= \max_{
\begin{subarray}{c}
K \subseteq \{ 1,\ldots,N \} \\
|K| = k
\end{subarray}
} \ \sum_{i \in K} D_i. 
\end{equation}

\begin{theorem} \label{thm:new_gap}
Let Assumption \ref{ass1} hold true.
It holds:
\begin{equation} \label{eq:gap_estim_sf}
J^* - \mathcal{J}^*
\leq
\frac{1}{2N^2} \,
\Bigg(
\max_{
\begin{subarray}{c}
Q \subseteq \{ 1,\ldots,N \} \\
|Q|= q \wedge N
\end{subarray}
} \
\sum_{i \in Q} \rho(\mathcal{Y}_i)^2
\Bigg)
\leq \frac{D[q \wedge N]}{2N^2}.
\end{equation}
\end{theorem}

Note that $D[N]= N C_1$, thus the new gap estimate is the same as the one obtained in Proposition \ref{prop:gap} when $q \geq N$ and it is strictly better when $q < N$.

\begin{proof}[Proof of Theorem \ref{thm:new_gap}]
We let the reader verify that $\nabla \tilde{f}$ is 1-Lipschitz. Then Lemma \ref{lemma:gap_estimate} and the homogeneity of $\rho$ yield
\begin{equation*}
J^* - \mathcal{J}^*
\leq \frac{1}{2} \rho(\mathcal{Y})^2
\leq \frac{1}{2N^2} \, \rho \Bigg( \sum_{i=1}^N \mathcal{Y}_i \Bigg)^2.
\end{equation*}
Applying \cite[Theorem 2]{Cassels1975}, we obtain that
\begin{equation*}
\rho \Bigg( \sum_{i=1}^N \mathcal{Y}_i \Bigg)^2
\leq
\max_{
\begin{subarray}{c}
Q \subseteq \{ 1,\ldots,N \} \\
|Q|= q \wedge N
\end{subarray}
} \,
\sum_{i \in Q} \rho(\mathcal{Y}_i)^2,
\end{equation*}
which proves the first inequality. Observing that
\begin{equation*}
\rho(\mathcal{Y}_i)^2
\leq d(\mathcal{Y}_i)^2
\leq \sum_{j=1}^M
d \big( \tilde{g}_{ij} ( \mathcal{X}_i ) \big)^2
= \sum_{j=1}^M \tilde{L}_j
d \big( g_{ij} ( \mathcal{X}_i ) \big)^2
= D_i,
\end{equation*}
we obtain the second inequality.
\end{proof}

\subsection{Duality and price of decentralization}

In this subsection we introduce a dual problem (we work again under Assumption \ref{ass2}) and investigate its connection with the geometric relaxed problem \eqref{pb:rel_geo}. This allows us to obtain a last refinement of the randomization gap.
For all $i=1,\ldots,N$ and for all $\lambda \in \mathcal{E}$, we introduce
\begin{align*}
\Phi_i(\lambda)=
\inf_{x_i \in \mathcal{X}_i} \langle \lambda, \tilde{g}_i(x_i) \rangle,
\qquad
\mathcal{Y}_i(\lambda)=
\underset{y_i \in \mathcal{Y}_i}{\text{argmin}} \, \langle \lambda, y_i \rangle,
\qquad
\mathcal{X}_i(\lambda)=
\underset{x_i \in \mathcal{X}_i}{\text{argmin}} \, \langle \lambda, \tilde{g}_i(x_i) \rangle.
\end{align*}
We refer to the following problem as the dual problem:
\begin{equation*} \tag{D} \label{pb:dual}
\sup_{\lambda \in \mathcal{E}}
\
\Big(
- \tilde{f}^*(\lambda)
+ \frac{1}{N} \sum_{i=1}^N \Phi_i(\lambda) \Big). 
\end{equation*}
Let $\mathcal{D}^*$ denote the value of Problem \eqref{pb:dual}.

\begin{assumption} \label{ass:closedness}
{The function $f \colon \mathcal{E}\rightarrow \mathbb{R}$ is lower semi-continuous and convex, and the set $\conv(\mathcal{Y})$ is closed. }
\end{assumption}

\begin{remark}
Assume that $\mathcal{E}$ is finite-dimensional.
If the sets $\mathcal{X}_i$ are compact and the maps $\tilde{g}_i$ continuous, then the sets $\mathcal{Y}_i= \tilde{g}_i(\mathcal{X}_i)$ are also compact. It is then easy to verify with Carathéodory's theorem that $\conv(\mathcal{Y}_i)$ is also compact, thus closed, which finally implies Assumption \ref{ass:closedness}.
\end{remark}

\begin{lemma}
The problem \eqref{pb:rel_geo} has a solution.
\end{lemma}

\begin{proof}
This is a direct application of \cite[Theorem 11.9]{bauschke2011convex}.
\end{proof}

The next lemma provides a duality result and a characterization of optimal solutions for problem \eqref{pb:rel_geo}.

\begin{lemma} \label{lemma:duality}
Let Assumptions \ref{ass1}, \ref{ass2}, \ref{ass3}, and \ref{ass:closedness} hold true.
Then, $\mathcal{J}^*= \mathcal{D}^*$ and the dual problem \eqref{pb:dual} has at least one solution.
Fix a solution $\lambda$ to Problem \eqref{pb:dual}. Let $y \in \mathcal{E}$. Then, $y$ is a solution to \eqref{pb:rel_geo} if and only if
$y \in \partial \tilde{f}^*(\lambda)$ and $ y \in \frac{1}{N} \sum_{i=1}^N \conv\big( \mathcal{Y}_i(\lambda) \big)$.
\end{lemma}

\begin{proof}
Let $h$ denote the indicatrix function of $\conv(\mathcal{Y})$. By Assumption \ref{ass1}, the domain of $\tilde{f}$ contains a neighborhood of $\conv(\mathcal{Y})$. By Assumption \ref{ass:closedness}, $h$ is lower semi-continuous. Therefore, the Fenchel-Rockafellar theorem \cite{Rockafellar2015} applies and yields
\begin{equation*}
\mathcal{J}^*
= \inf_{y \in \mathcal{E}} \, \Big( f(y) + h(y) \Big)
= \sup_{\lambda \in \mathcal{E}}
\Big( -\tilde{f}^*(\lambda) - h^*(-\lambda) \Big).
\end{equation*}
Moreover, the supremum in the right-hand side is a maximum.
We have
\begin{equation*}
-h^*(-\lambda)
= \inf_{y \in \conv(\mathcal{Y})} \,
\langle \lambda, y \rangle
=
\inf_{y \in \mathcal{Y}} \, \langle \lambda, y \rangle
= \frac{1}{N} \sum_{i=1}^N \Phi_i(\lambda).
\end{equation*}
As a consequence, $\mathcal{J}^*=\mathcal{D}^*$ and problem \eqref{pb:dual} has at least one solution.

Now let us fix a solution $\lambda$ to the dual problem \eqref{pb:dual}.
Let $y \in \mathcal{E}$.
Then $y$ is a solution if and only if (i) $\tilde{f}(y) + \tilde{f}^*(\lambda) =
\langle \lambda, y \rangle$ and (ii) $h(y) + h^*(-\lambda)= - \langle \lambda, y \rangle$.
The condition (i) is equivalent to $y \in \partial \tilde{f}(\lambda)$. The condition (ii) is equivalent to
\begin{equation*}
y \in \conv(\mathcal{Y}) \text{ and }
\langle \lambda, y \rangle
= - h^*(-\lambda)
= \inf_{y' \in \mathcal{Y}} \langle \lambda, y' \rangle.
\end{equation*}
Thus $\text{(ii)} \Longleftrightarrow y \in Y$, where
$Y = \underset{y' \in \conv(\mathcal{Y})}{\text{argmin}} \langle \lambda, y' \rangle$. We further have
\begin{equation*}
Y = \conv \Bigg(
\underset{y' \in \mathcal{Y}}{\text{argmin}} \ \langle \lambda,y' \rangle \Bigg)
= \conv \Big( \, \frac{1}{N} \sum_{i=1}^N \mathcal{Y}_i(\lambda) \Big)
= \frac{1}{N} \sum_{i=1}^N \conv \big( \mathcal{Y}_i(\lambda) \big),
\end{equation*}
which concludes the proof.
\end{proof}

\begin{remark}
If $\tilde{f}$ is differentiable on $\mathcal{E}$, with a Lipschitz-continous gradient, then $\tilde{f}^*$ is strongly convex (see \cite[Theorem 18.15]{bauschke2011convex}), which implies that \eqref{pb:dual} has a unique solution.
\end{remark}

Let us fix a solution $\lambda$ to the dual problem until the end of the subsection. Let us consider
\begin{equation*}
J_{\text{dec}}
= \inf_{x \in \mathcal{X}} J(x), \quad
\text{subject to: } x_i \in \mathcal{X}_i(\lambda), \quad \forall i=1,\ldots, N.
\end{equation*}
In words, we restrict $\mathcal{X}_i$ to the best-responses corresponding to the dual variable $\lambda$.
Following the terminology of \cite{Wang2017}, we call price of decentralization the real number $p= J_{\text{dec}}- J^*$.

\begin{proposition} \label{prop:ultimate_bound}
Let Assumptions \ref{ass1}, \ref{ass2}, \ref{ass3}, and \ref{ass:closedness} hold true.
It holds:
\begin{equation*}
p \leq J_{\mathrm{dec}} - \mathcal{J}^*
\leq \frac{1}{2N^2} \Bigg( \max_{
\begin{subarray}{c}
Q \subseteq \{ 1,\ldots, N \} \\
| Q |= q \wedge N
\end{subarray}
} \ \sum_{i \in Q} \rho \big( \mathcal{Y}_i(\lambda) \big)^2 \Bigg).
\end{equation*}
\end{proposition}

\begin{proof}
The definition of $J_{\text{dec}}$ and Lemma \ref{lemma:duality} respectively yield:
\begin{equation*}
J_{\text{dec}}
= \inf_{y \in \frac{1}{N} \sum_{i=1}^N \mathcal{Y}_i(\lambda)} \, \tilde{f}(y)
\qquad
\text{and}
\qquad
\mathcal{J}^*
= \inf_{y \in \frac{1}{N} \sum_{i=1}^N \conv ( \mathcal{Y}_i(\lambda) )} \, \tilde{f}(y).
\end{equation*}
The announced estimate follows then from Lemma \ref{lemma:gap_geo} and \cite[Theorem 2]{Cassels1975}, as in the proof of Theorem \ref{thm:new_gap}.
\end{proof}

\begin{remark}
The randomization gap is bounded from above by $J_{\mathrm{dec}}- \mathcal{J}^*$.
Moreover, one can show that $\rho(\mathcal{Y}_i(\lambda)) \leq \rho(\mathcal{Y}_i)$. Thus Proposition \ref{prop:ultimate_bound} provides a last refinement of the gap estimate \eqref{eq:gap_estim_sf}.
\end{remark}

\section{Comments on numerical aspects and examples} \label{sec:comments}

\subsection{Literature comparison}

Let us compare our results and our method with the work of Wang \cite{Wang2017}. Our gap estimate, as well as our estimate of the price of decentralization, are of order $\mathcal{O}(\min(q,N)/N^2)$, while the estimates obtained by applying \cite[Theorem 3.5]{Wang2017} are of order $\mathcal{O}(q^2/N^2)$. We emphasize that our first gap estimate, of order $\mathcal{O}(1/N)$, already improves \cite{Wang2017} when $q \gg \sqrt{N}$. Note that the geometric relaxation employed in Section \ref{subsec-refined_gaps} is the same as the one used in \cite{Wang2017}.

Let us compare our algorithmic approaches. At a general level, one can observe that we have a primal approach, while Wang solves the dual problem to the relaxed problem. Our approach is restricted to the case where $f$ is differentiable, while the dual approach allows to tackle the case of hard constraints (for example when $f$ is the indicator function of some convex set).
Both approaches leverage the decomposability of the problem into $N$ problems and require that the subproblems can be easily solved. Let us emphasize however that we only need to be able to compute a single solution for those problems, while \cite[Algorithm 2]{Wang2017} requires to compute the full set of $\xi$-optimal solutions, which may be much more difficult.
Our algorithm does not require to perform Shapley-Folkman decompositions, contrary to \cite{Wang2017}. This is a major advantage when the dimension of the aggregate $q$ is very large.
Also, we do not need to evaluate $f^*$.
As a counterpart, we are only able to find $\mathcal{O}(1/N)$-optimal solutions, while the algorithm of \cite{Wang2017} can find $\mathcal{O}(q^2/N^2)$-optimal solutions. The design of a method for the computation of $\mathcal{O}(q \wedge N /N^2)$-solutions will be the topic of future research.

\subsection{Social welfare example}\label{sec:social}

A particularly interesting instance of \eqref{pb:aggre} is the social welfare optimization problem investigated in a closely related paper by Mengdi Wang \cite{Wang2017}. The cost function is the following:
\begin{equation} \label{pb:Wang}
\inf_{x_i\in \mathcal{X}_i} f_0 \left(\frac{1}{N}\sum_{i=1}^{N}h_{i}(x_i)\right) + \frac{1}{N}\sum_{i=1}^{N}l_{i}(x_i).
\end{equation}
Following her terminology, the function $h_{i}$ is the contribution of  agent $i$ to some common goods, $f_0$ is a social cost function of the common goods, and $l_{i}$ describes the individual preference of agent $i$. There are various applications fitting into the framework of \eqref{pb:Wang}, see \cite{Wang2017}.
In particular, some power system management problems can be modeled as \eqref{pb:Wang}. Such a problem is investigated in \cite{Seguret2020}: $x_i$ represents the production profile of the generator $i$, $l_{i}(x_i)$ is its individual production cost, $f_0$ denotes the demand elasticity or, equivalently, a penalty function that depends on the difference between the average production and some inflexible demand $D$ (e.g.\@ $f_0:=\Vert \cdot-D\Vert^2$) so as to penalize the deviation of the overall production from the inflexible demand.

{Let us also mention the \emph{resource allocation problems}, investigated in \cite{beaude2020privacy}, for example. These problems are of the form \eqref{pb:Wang}, where $f_0$ is the indicator function (as defined in \cite[Example 1.25]{bauschke2011convex}) of a given point $y \in \mathcal{E}$, modelling the resource to be allocated over the agents. These problems find applications in energy management (see for example \cite{ghaderyan2021fully} and \cite{jacquot2018analysis}).
They do not fit to the current framework, since the indicator function is not differentiable, but can be reasonably well approximated, replacing the indicator by a penalty function.
}

\color{black}
\subsection{{Discussion on the case of finite feasible sets}} \label{subsec:discussion}

The stochastic Frank-Wolfe algorithm investigated in the previous sections was motivated by the difficulty of manipulating probability measures, from a numerical point of view. However, when the sets $\mathcal{X}_i$ are finite, with relatively low cardinality, it is possible to store probability measures with possibly full support and some other numerical methods can be used to solve the randomized problem.
Let us assume (in this subsection only) that the sets $\mathcal{X}_i$ are of cardinality $n_i \in \mathbb{N}$ and that $\mathcal{X}_i = \{\textbf{x}_i^1,\ldots,\textbf{x}_i^{n_i} \}$. Then the randomized problem reads:
\begin{equation} \label{pb:simplex2}
\min_{\nu =(\nu_1,\ldots,\nu_N)} \
f \Big( \frac{1}{N} \sum_{i=1}^N \sum_{\ell=1}^{n_i} \nu_i^{\ell} g_i(\textbf{x}_i^{\ell})\Big),
\quad
\text{subject to: } \nu_i \in \Delta(n_i),
\end{equation}
where $\Delta(n_i)$ denotes the $(n_i-1)$-simplex, i.e.
\begin{equation*}
\Delta(n_i)= \Big\{ \nu \in \R^{n_i} \, \Big| \,
\sum_{\ell=1}^{n_i} \nu^{\ell} = 1 \text{ and }
\nu^{\ell} \geq 0,\, \forall \ell =1,\ldots, n_i
\Big\}.
\end{equation*}
The problem is a convex program on a Cartesian product of $N$ simplices.
Let us first note that in this framework, Assumption \ref{ass3} is trivially verified, since problem \eqref{eq:sub_pb_i} is just a minimization problem over $\mathcal{X}_i$ which can be solved by enumeration. Moreover
any variant of the Frank-Wolfe algorithm can be implemented, in order to solve the randomized problem in a faster way. We refer the reader to \cite{Jaggi2013,lacoste2015global}.
Some other methods could also be implemented. The problem could be solved with the projected gradient descent algorithm, but the projection on the simplices is expensive (see \cite{condat2016fast}). Instead, the problem can be naturally addressed with the mirror descent algorithm \cite{beck2003mirror} (see in particular the entropic descent algorithm in Section 5), and with accelerated versions of the entropic descent algorithm \cite{krichene2015accelerated}.

Let us observe that if we require $\nu$ to have integer entries in the problem \eqref{pb:simplex2}, then we are back to the original problem. Indeed, the elements of the simplex with integer entries are its vertices, that is, the vectors of the form $(0,\ldots,0,1,0,\ldots,0)$. Therefore the original problem can be viewed as a mixed-integer convex program (MICP) and can be addressed numerically with combinatorial techniques, see \cite{bonami2012algorithms,coey2020outer} and the references therein.

\subsection{Aggregative optimal control}

We describe here a large-scale optimal control problem of the form of problem \eqref{pb:aggre}, with an infinite-dimensional aggregate space. We verify Assumptions \ref{ass1}, \ref{ass2}, and \ref{ass3} and we discuss the applicability of the Stochastic Frank-Wolfe algorithm.

Let us first fix the data of the problem. For any $i=1,\ldots,N$, we consider: an initial condition $z_i^0 \in \R^n$, a control set $U_i \subseteq \R^m$, a dynamics $F_i \colon (z_i,u_i) \in \R^n \times U_i \mapsto F_i(z_i,u_i) \in \R^n$, and a contribution function $\phi_i \colon \R^n \times U_i \rightarrow \R^k$. We also consider a social cost $\ell \colon \R^k \rightarrow \R$.
We make the following assumptions:
\begin{enumerate}
\item \emph{Regularity and boundedness.}
For any $i=1,\ldots,N$,
\begin{itemize}
\item $U_i$ is non-empty and compact
\item $F_i$ is continuous, Lipschitz continuous with respect to $z_i$, uniformly with respect to $u_i$; moreover, there exists a constant $K_i$ such that $\| F_i(z_i,u_i) \| \leq K_i (1 + \| z_i \| )$, for any $(z_i,u_i) \in \R^n \times U_i$
\item $\phi_i$ is continuous; moreover, there exists a function $R_i \colon \R_+ \rightarrow \R_+$ such that $\| \phi_i(z_i,u_i) \| \leq R_i \big( \| z_i \| + \| u_i \| \big)$,
for any $(z_i,u_i) \in \R^n \times U_i$.
\end{itemize}
\item \emph{Regularity of the social cost.}
The function $\ell$ is continuously differentiable, moreover, $\ell$ and $\nabla \ell$ are Lipschitz continuous with moduli $L_{\ell}$ and $L_{\nabla \ell}$, respectively.
\item \emph{Convexity assumption.}
For any $i=1,\ldots,N$, for any $y \in \R^k$, for any $z_i \in \R^n$, we define $\mathcal{Z}_i(y,z_i)$ the set of all elements $(\bar{z}_1,\bar{z}_2)$ in $\R^{n+1}$, where there exists $u_i\in U_i$, such that $\bar{z}_1 = F_i(z_i,u_i) $ and $\bar{z}_2 \geq \langle \nabla  \ell(y), \phi_i(z_i,u_i) \rangle$. The set $\mathcal{Z}_i(y,z_i)$ is convex.
\end{enumerate}

Let us mention a particular case in which the above convexity assumption is true:
for any $i=1,\ldots,N$, for any $y \in \R^k$, for any $z_i \in \R^n$,
 \begin{itemize}
 \item For any $z_i$, the map $u_i \mapsto F_i(z_i,u_i)$ is affine.
\item The set $U_i$ is convex and the function $u_i \in U_i \mapsto \langle \nabla  \ell(y), \phi_i(z_i,u_i) \rangle$ is convex.
\end{itemize}

For any $i =1,\ldots,N$, consider the set $\mathcal{X}_i$ of pairs $(z_i,u_i) \in W^{1,\infty}(0,T;\R^n) \times L^\infty(0,T;\R^m)$ satisfying
\begin{equation*}
\dot{z}_i(t)= F_i(z_i(t),u_i(t)), \quad
z_i(0)= z_i^0, \quad
u_i(t) \in U_i, \quad \text{for a.e.\@ $t \in (0,T)$}. 
\end{equation*}
A direct application of Gronwall's lemma shows that for any $(z_i,u_i) \in \mathcal{X}_i$, we have $\|  z_i \|_{L^\infty(0,T;\R^n)} \leq \tilde{K}_i$, where $\tilde{K}_i=(1+ \| y_0^i \|) \exp(K_i T) - 1$.

The aggregative optimal control problem of interest is defined as follows:
\begin{equation} \label{eq:aggreg_oc}
\inf_{(z_i,u_i)_{i=1}^N \in \prod_{i=1}^N \mathcal{X}_i} \
\int_0^T \ell \Big( \frac{1}{N} \sum_{i=1}^N \phi_i \big( z_i(t),u_i(t) \big) \Big) \, \text{d}t.
\end{equation}
It is a special case of problem \eqref{pb:aggre} with $m=1$, $\mathcal{E}_1= \mathcal{E}= L^2(0,T;\R^k)$, and
\begin{align*}
\begin{array}{rccl}
g_i \colon & \! \! \! \! (z_i,u_i) \in \mathcal{X}_i
& \mapsto & \big( t \in (0,T) \mapsto \phi_i(z_i(t),u_i(t)) \big) \in L^2(0,T;\R^k) \\[0.7em]
f \colon & \! \! \! \! y \in L^2(0,T;\R^k)
& \mapsto & \int_0^T \ell(y(t)) \, \text{d}t.
\end{array}
\end{align*}
Problem \eqref{eq:aggreg_oc} can be seen as a nonconvex optimal control problem with state variable $(z_i)_{i=1}^N$. It finds application in energy management, in the situations mentioned in the introduction and in particular those involving storage devices, for which the dynamics of the state-of-charge must be taken into account. Once again we refer the reader to \cite{Seguret2020}, which considers a convex stochastic aggregative optimal control problem. In general, only dynamic-programming-based methods can provide global solutions to nonlinear optimal control problems. They are not applicable here because of the high dimension of the state variable, equal to $Nn$.

It is easy to verify that $\nabla f$ is continuously differentiable and that $f$ and $\nabla f$ are Lipschitz-continuous with moduli $\sqrt{T} L_{\ell}$ and $L_{\nabla \ell}$, respectively. Let $\hat{K}_i$ be an upper bound of $\sup_{u_i \in U_i} \, \| u_i \|$, for all $i \in 1,\ldots, N$. Then $g_i(\mathcal{X}_i)$ is bounded in $L^2(0,T;\R^k)$, with diameter bounded by $2\sqrt{T} R_i(\tilde{K}_i + \hat{K}_i)$.
Therefore, Assumption \ref{ass1} is satisfied.
If $\ell$ is convex, then $f$ is also convex and then Assumption \ref{ass2} holds true. Let us verify Assumption \ref{ass3}. Given $y \in G(\mathcal{X})$, the problem \eqref{eq:sub_pb_i} to be solved at each iteration of the SFW algorithm reads
\begin{equation} \label{}
\inf_{(z_i,u_i) \in \mathcal{X}_i} \
\int_0^T \big\langle \nabla \ell(y(t)), \phi_i(z_i(t),u_i(t)) \big\rangle \, \text{d}t.
\end{equation}
This is an optimal control with state variable $z_i$, which falls into the class of problems introduced in \cite[Chapter III, Theorem 4.1]{fleming1975} and therefore possesses a solution. It the dimension of the state variable, $n$, is small, then it can be solved by dynamic programming. We refer the reader to \cite{falcone2013semi}.

\subsection{Supervised learning problems}

We describe and discuss here two applications of problem \eqref{pb:aggre} in the context of supervised learning.

\paragraph{Neural networks with one hidden layer}

We refer the reader to \cite{Chizat2018,mei2018mean,mei2019mean}. Consider a neural network of the form
$\frac{1}{N} \sum_{i=1}^N \sigma_*(\mathbf{a},x_i)$, where $\mathbf{a} \in \R^d$ is the feature vector, $x= (x_i)_{i=1}^N \in (\R^{D})^N$ are the network parameters (to be optimized), and $\sigma_* \colon \R^d \times \R^D \rightarrow \R$ an activation function. We consider a loss function $\varphi \colon \R \rightarrow \R_+$. Given a data set $(\mathbf{a}_j,b_j)_{j=1}^M \in (\R^d \times \R)^M$, the learning problem of interest writes
\begin{equation}\label{pb:Mei}
    \inf_{(x_i)_{i=1}^N \in (\mathbb{R}^{D})^N} \frac{1}{M}\sum_{j=1}^M \varphi \Big(b_j - \frac{1}{N}\sum_{i=1}^N \sigma_{*}(\mathbf{a}_j,x_i)\Big).
\end{equation}
It is of the form \eqref{pb:aggre}, with $\mathcal{E}=\R^M$, $\mathcal{E}_j= \R$, $f_j(y_j)= \varphi(b_j-y_j)/M$, $g_{ij}(x_i)= \sigma_*(\mathbf{a}_j,x_i)$. Assume that the set
$\{ \sigma_*(\textbf{a}_j,x) \, | \, x \in \R^D, \, j \in \{ 1,\ldots,M \} \}$
has a bounded diameter $\bar{d}$. Assume moreover that $\varphi$ is continuously differentiable and that $\nabla \varphi$ is $L_{\nabla \varphi}$-Lipschitz continuous. Then Assumption \ref{ass1} is satisfied and we have $D_i= L_{\nabla \varphi} \bar{d}^2$, for the coefficients $D_i$ introduced in \eqref{eq:D_i}. Therefore, by Theorem \ref{thm:new_gap}, the optimality gap is bounded by
\begin{equation*}
\frac{(M \wedge N) L_{\nabla \varphi} \bar{d}^2}{2N^2}.
\end{equation*}
Moreover, if $\varphi$ is convex, then Assumption \ref{ass2} holds true. The resolution of the subproblems \eqref{eq:sub_pb_i} is not  easy in general, we refer the reader to \cite{d2020global} where the linearized problems are shown to be solvable by second-order cone programming in the case of ReLu activation functions.

Note that we are here in the symmetric case, as defined at the end of Section \ref{subsec:relax}. The mean-field relaxation proposed in Lemma \ref{lm:mean-field} was also utilized in \cite{mei2019mean,mei2018mean} for learning problems of the form \eqref{pb:Mei}. A gap estimate of order $\mathcal{O}(1/N)$ is demonstrated, in the case of a quadratic loss function $\varphi$, see \cite[Prop.\@ 1]{mei2018mean}. Our gap estimate is more general since $\nabla \varphi$ is only supposed to be Lipschitz; moreover, it is more precise in the case of an overparametrized network (i.e.\@ when $M<N$), since then it is of order $\mathcal{O}(M/N^2)$.

\paragraph{Sparse reconstruction}

Another important learning example is the sparse reconstruction with the $\ell_0$-penalty, see \cite{Mallat2009,Mairal2014}. Let $D$ be a $M$ by $N$ dictionary matrix. The objective is to approximate the observed vector $x\in\mathbb{R}^M$ by a sparse linear combination of the columns of $D$. Following \cite[Eq.\@ 5.6]{Mairal2014}, we are interested in the following least square problem with the $\ell_0$-penalty:
\begin{equation*}
\inf_{\alpha \in \mathbb{R}^{N}}
\frac{1}{2} \big\|  x- D \alpha \big\|^2 + \beta \|\alpha\|_{\ell_0}
=
\inf_{\alpha \in \mathbb{R}^{N}}
\frac{1}{2} \sum_{j=1}^M  \Big( x_j - \sum_{i=1}^N D_{ji}\alpha_i\Big)^2
+
\beta \sum_{i=1}^N \mathbf{1}_{\R \backslash \{ 0 \}} (\alpha_i),
\end{equation*}
where $\beta$ is a constant and $\|\alpha\|_{\ell_0}$ counts the number of non-zero entries in a vector $\alpha$. Adding constraints of the form $\alpha_i \in [u_i,v_i]$ to the problem, it is easy to see that Assumptions \ref{ass1} and \ref{ass2} are satisfied. The subproblems \eqref{eq:sub_pb_i} are here of the form
\begin{equation*}
\inf_{\alpha_i \in [u_i,v_i]} z \alpha_i + \mathbf{1}_{\R \backslash \{ 0 \} }(\alpha_i)
\end{equation*}
for some real number $z$. One can show that there is a solution that necessarily lies in $\{ u_i, v_i, 0 \}$, thus it is easy to compute.

Finally, let us mention other applications of the problem \eqref{pb:aggre} in a convex framework, for instance, the ``sharing problem" in \cite{Boyd2011}, Lasso regression in \cite{Fercoq2016} and the dual problem of a linear support vector machine (SVM) in \cite{shalev2009stochastic,Fercoq2016}.

\color{black}

\section{Numerical test} \label{sec:num}

In this section we provide numerical results for a mixed-integer linear quadratic problem of the form \eqref{pb:aggre}.
Let $A$ be a real $M\times N$ matrix and let $\bar{y}\in \R^M$. Consider the following problem:
\begin{equation*} \label{pb:miqp1}\tag{MIQP}
\min_{x \in \{ 0, 1 \}^N}
J(x)
:= \frac{1}{N^2} \|Ax -\bar{y}\|^2_{\mathbb{R}^M}
= \sum_{j=1}^M \left(\frac{1}{N}\sum_{i=1}^{N} A_{ji}x_i - \frac{\bar{y}_{j}}{N}\right)^2.
\end{equation*}
Problem \eqref{pb:miqp1} has the form \eqref{pb:aggre}, with $f_j(y_j) =\big( y_j-\frac{\bar{y}_{j}}{N} \big)^2$ for $1\leq j\leq M$, and $g_{ij}(x_i) = A_{ji}x_i$ for $1\leq i\leq N$, $1\leq j\leq M$. Moreover, Assumption \ref{ass1} is satisfied with $\tilde{L}_j=2$ and $d_{ij}= |A_{ji}|$. Thus $C_1= \frac{2}{N} \sum_{i=1}^N \sum_{j=1}^M |A_{ji}|$. Due to the linearity of $g_{ij}$, the randomized problem coincides with the minimization problem of $J$ on $[0,1]^N$, which is a convex linear-quadratic program that can be solved with independent methods; thus it is easy here to obtain a precise estimate of $\mathcal{J}^*$.

{In the numerical simulation, we draw the parameters $A_{ji}$ according to the uniform distribution on the interval $[0,1]$ while $y_j$ is drawn according to the uniform distribution on $[0,N/2]$. Thus, $C_1 \approx M$ and the gap estimate is given by $\frac{C_1}{2N} \approx 0.5$. We perform our numerical experiments on a laptop with one Intel Core i5-8250U processor (4 cores) at 1.60 GHz and 8 GB RAM.}

{
The first experiment is a comparison of Algorithm \ref{alg1+k} with an open source solver, SCIP, \cite{scip2021} and a commercial solver, GUROBI, \cite{gurobi2018gurobi}. As mentioned before, the dual (randomized) problem is a convex linear-quadratic program. We can compute $\mathcal{J}^{*}$ easily by solver GUROBI. Table \ref{tab1new} shows the value $\mathcal{J}^{*}$ and results of \eqref{pb:miqp1} obtained from SCIP, GUROBI and Algorithm \ref{alg1+k}, for different values of $M,N$ ranging from 100 to 3200. In Table \ref{tab1new}, ``Nan" indicates that the solver has failed to return a result or that computation time has exceeded one hour. Denote by $v_{s}$ the result of Algorithm \ref{alg1+k}. The indicated gap is a relative gap, in percent, defined by $(v_s - \mathcal{J}^{*})/\mathcal{J}^{*}$. We can observe that the relative gap 
decreases as $N$ increases, which is consistent with the randomized gap \eqref{eq:gap}. The last three columns of Table \ref{tab1new} show that Algorithm \ref{alg1+k} is competitive in terms of execution time, in comparison with SCIP and GUROBI. Finally, observe that for $N=M=3200$, none of the two solvers could solve the problems while Algorithm \ref{alg1+k} has provided a solutions in approximately 6 minutes.
}

\begin{table}[h]
\begin{center}
\begin{tabular}{|r||r|r|r|rr||r|r|r|}
\hline
\multirow{2}{*}{$N=M$} & \multirow{2}{*}{$\mathcal{J}^{*}$\ \ } & \multicolumn{1}{|c|}{SCIP} & \multicolumn{1}{|c|}{GUR.} & \multicolumn{2}{c||}{SFW} & SCIP & GUR. & SFW            \\ \cline{3-9} 
                     &                             &           \multicolumn{1}{|c|}{value}           &       \multicolumn{1}{|c|}{value}                 & \multicolumn{1}{c|}{value} & \multicolumn{1}{c||}{
                     \hspace{-4mm} \begin{tabular}{c} gap \hspace{-2mm}\\ in \% \hspace{-2mm} \\ \end{tabular} \hspace{0mm}} & \multicolumn{3}{c|}{time in seconds } \\ \hline
100  &                2.077                 &     2.077                  & 2.077                         & \multicolumn{1}{r|}{2.136}      &     2.870       &    0.88   &    0.20    &  0.03        \\
200                  &        4.120                     &       4.120                &               4.120          & \multicolumn{1}{r|}{4.159}      &    0.956  &    5.99 &   0.69   &   0.09         \\
400                  &       7.871                      &      7.871                 &                    7.871     & \multicolumn{1}{r|}{7.904}      &    0.430  &   87.78   &   7.90    &   0.91   \\
800                  &           15.953                  &         Nan              &             15.954            & \multicolumn{1}{r|}{15.966}      &     0.079   &   Nan   &    10.63    &   6.18      \\
1600                 &     32.045                      &      Nan                 &            32.048             & \multicolumn{1}{r|}{32.0585}      &     0.042   &  Nan   &   81.41     &  42.51      \\
3200                 &      64.717                       &        Nan               &               Nan          & \multicolumn{1}{r|}{64.724}      &   0.012  &  Nan   &   Nan     &   330.95 \\
\hline
\end{tabular}
\caption{Comparison of the approximate values and execution times obtained with SCIP, GUROBI and Algorithm \ref{alg1+k} for problem \eqref{pb:miqp1} with $M=N = $ 100, 200, 400, 800, 1600 and 3200. In Algorithm \ref{alg1+k}, we take $n_k=1$ and $K = 2N$ iterations.}
\label{tab1new}
\end{center}
\color{black}
\end{table}

{The second experiment is on the basic Frank-Wolfe algorithm \ref{alg1} and its stochastic version \ref{alg1+k}. In this experiment, we fix $M=N=1000$.}
Figure \ref{fig:MIQP1} shows the outcome of the basic Frank-Wolfe algorithm \ref{alg1} with $200$ iterations. The left sub-figure shows the evolution of $\gamma_k$ for $\omega_k= 2/(k +2)$ (green curve) and for $\omega_k$ determined by line search \eqref{eq:ls1} (red curve). A sub-linear rate of convergence is observed (note that logarithmic scales are employed for both axes). The right sub-figure represents the evolution of $J(X^k)-\mathcal{J}^*$, where $X^k$ is a random variable with distribution $\mu^k$. For both choices of $\omega_k$, approximate solutions to the problems are simulated, with a gap smaller than $10^{-3}$, significantly smaller than the gap estimate $\frac{C_1}{2N}$. The line search approach is quicker than the approach with $\omega_k= \frac{2}{k+2}$.

\begin{figure}[ht]
	\centering
	\includegraphics[width=1\linewidth]{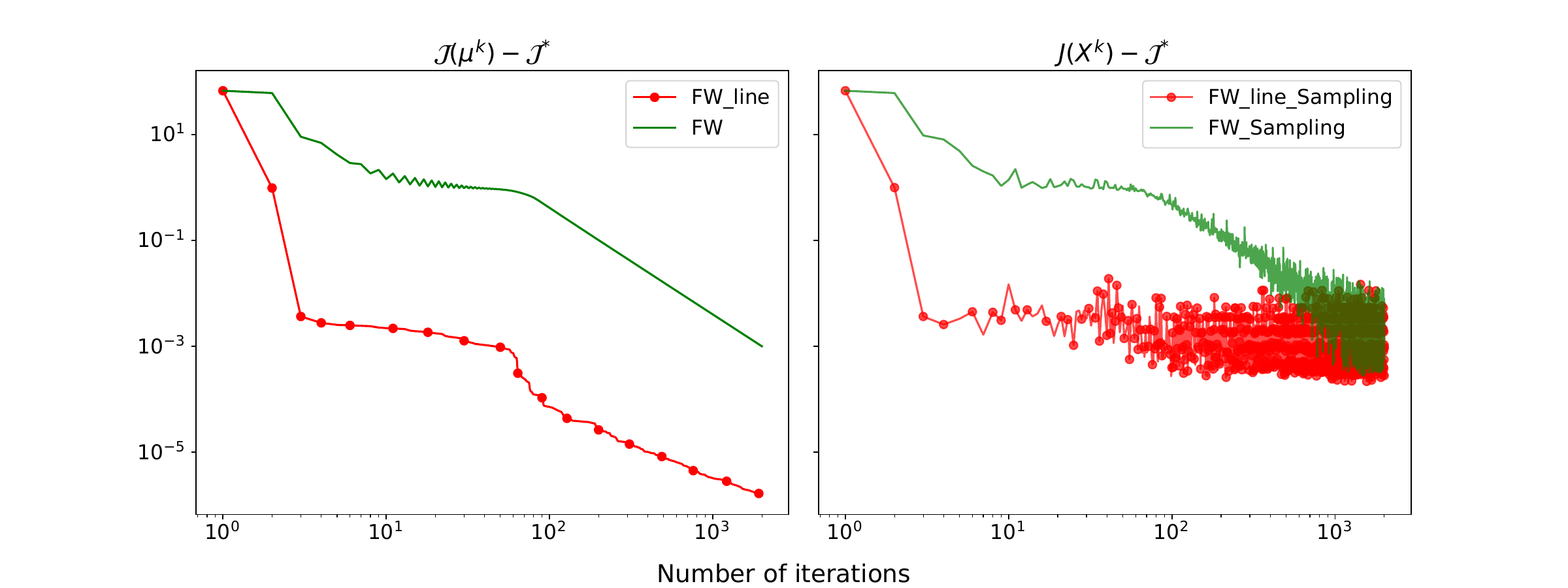}
	\caption{MIQP by Algorithm \ref{alg1}, $2000$ iterations, with $\omega_k = 2/(k+2)$ and line search \eqref{eq:ls1}.}
	\label{fig:MIQP1}
\end{figure}

Figure \ref{fig:MIQP2} shows the outcome of Algorithm \ref{alg1+k} (with the modification suggested in Remark \ref{rm:modif}), for different (constant) choices of $n_k$ with 200 iterations, for two different stepsize rules ($\omega_k= 2/(k+2)$ on the left, line search on the right). Since the algorithm is stochastic, we have tested it 50 times to evaluate its efficiency; the curves represent the average value of $\gamma_k$. The standard deviation (for these 8 instances of the SFW method) is displayed on Figure \ref{fig:MIQP3}. In all cases, an average value of the gap significantly smaller than $\frac{C_1}{2N}$ can be reached; the standard deviation is also significantly smaller than $\frac{C_1}{2N}$ at the last iterations. There is a benefit (both in expectation and standard deviation) in increasing the number of simulations $n_k$ (note that the choice $n_k=1000$ is much smaller the rule suggested by Corollary \ref{coro:convergence}). Yet the convergence is slower in comparison with the basic Franck-Wolfe algorithm, which can be explained by the use of the selection method at each iteration.
\begin{figure}[ht]
	\centering
	\includegraphics[width=1\linewidth]{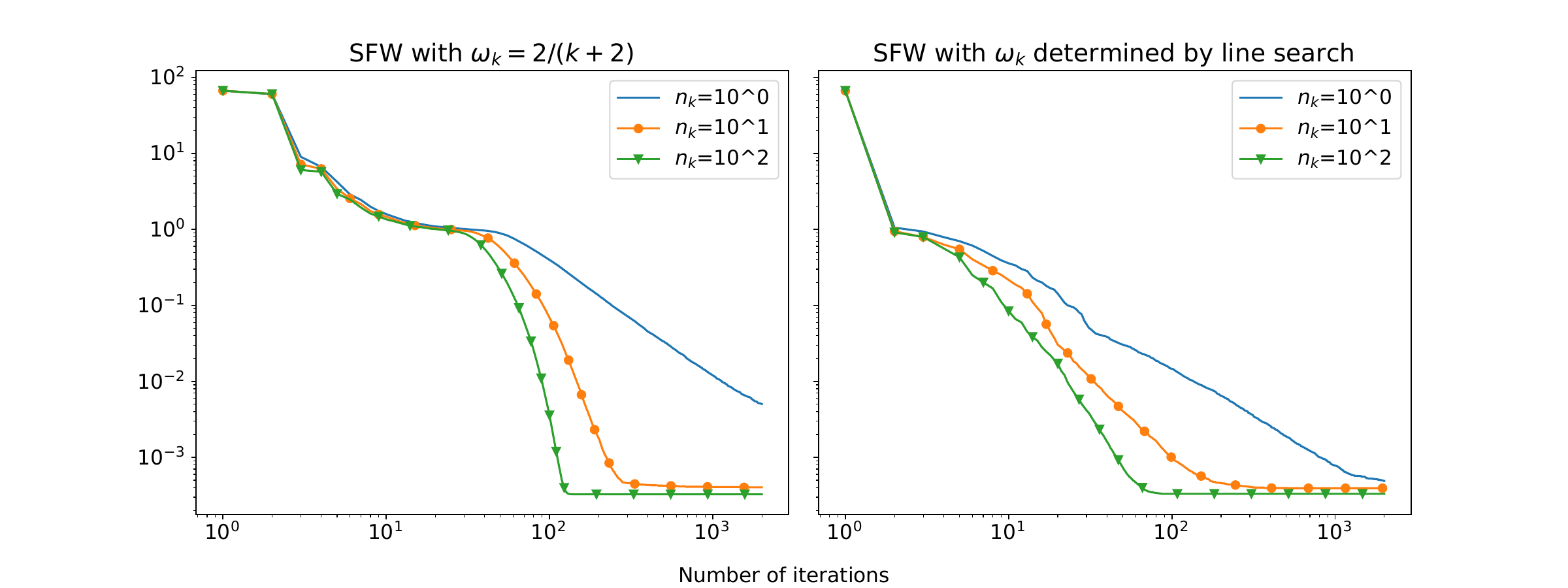}
	\caption{MIQP by Algorithm \ref{alg1+k} with $2000$ iterations, expectation of the gap.}
	\label{fig:MIQP2}
	\centering
	\includegraphics[width=1\linewidth]{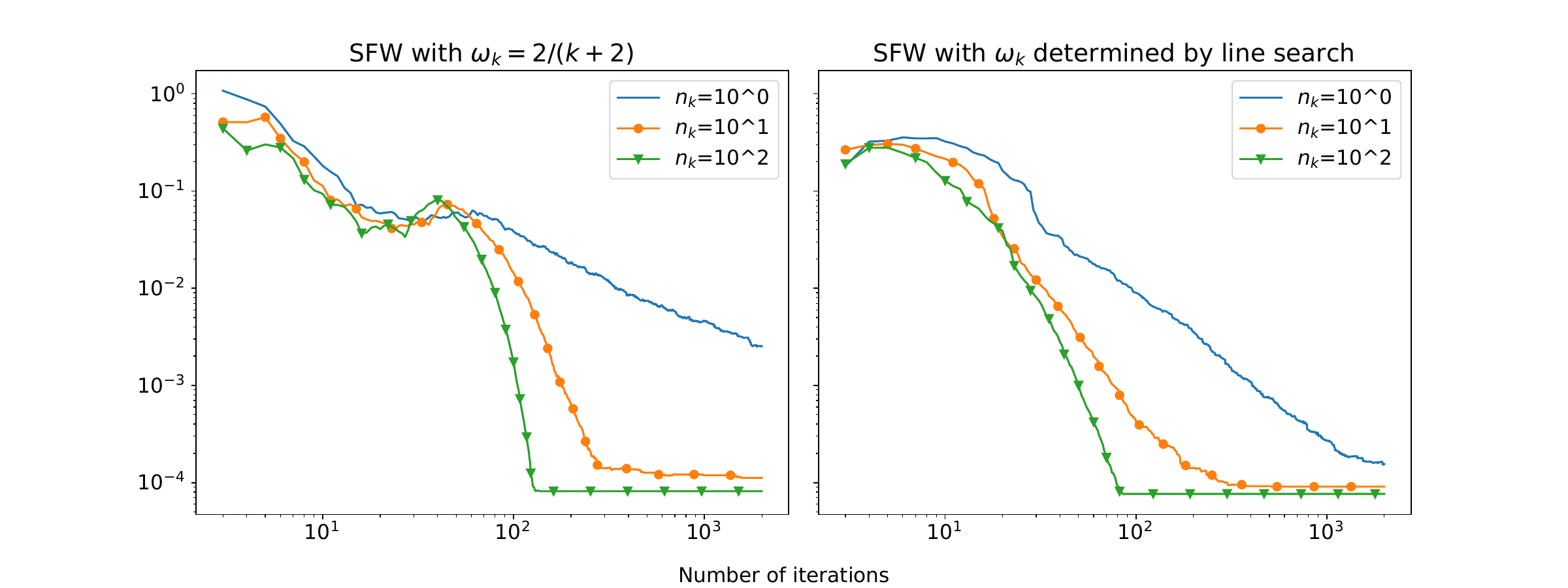}
	\caption{MIQP by Algorithm \ref{alg1+k} with $2000$ iterations, standard deviation of the gap.}
	\label{fig:MIQP3}
\end{figure} 

\section{Conclusion}

We have investigated a large-scale and aggregative optimization problem and its relaxation. New error bounds for the relaxation gap have been obtained. We have proposed a tractable algorithm for its resolution with a detailed convergence analysis relying on  concentration inequalities. Assuming that an efficient method for the resolution of the subproblems is available, the implementation of our stochastic Frank-Wolfe method is easy. 

Future research will focus on refinements of the selection method, allowing the computation of $\mathcal{O}(q \wedge N/N^2)$-solutions. We also aim at working on more complex problems, involving for example convex constraints on the aggregate, {as for example the resource allocation problems mentioned in the introduction. Such constraints could be handled with extensions of the Frank-Wolfe algorithm for non-smooth costs as those proposed in \cite{silveti2020generalized,yurtsever2019conditional2}.}
Finally, we intend to apply our method to large-scale optimal control problems, such as nonconvex variants of the problem investigated in \cite{Seguret2020}.

\appendix

\section{Concentration inequalities and other technical lemmas}

\begin{proposition} \label{prop:exp}
Consider $T$ real-valued random variables $(Y_t)_{t=1,...,T}$. Let $(\mathcal{F}_t)_{t=1,...,T}$ denote the associated filtration ($\mathcal{F}_0$ is the trivial $\sigma$-algebra). Let $Z_t= \mathbb{E}[Y_t^2 | \mathcal{F}_{t-1}]$ and let $S_T= \sum_{t=1}^T Y_t$.
Assume the following:
\begin{equation} \label{eq:prop_concentration}
(i) \quad \mathbb{E}[Y_t \mid \mathcal{F}_{t-1}] = 0, \qquad
(ii) \quad Y_t \leq m, \qquad
(iii) \quad \sum_{t'=1}^T Z_{t'} \leq v, \quad \text{a.s.}
\end{equation}
for all $t=1,...,T$ and for some constants $m$ and $v$. Then, $\mathbb{E}[S_T^2] \leq v$.
Moreover, for any $\epsilon > 0$,
\begin{equation}
\mathbb{P} \big[ S_T \geq \epsilon \big] \leq \exp \Big( -\frac{\epsilon^{2}}{2\left(v+ \epsilon m / 3\right)} \Big).
\end{equation}
\end{proposition}

\begin{proof}
The estimate of $\mathbb{E}[S_T^2]$ can be easily obtained by induction.
For the estimate of $\mathbb{P} \big[ S_T \geq \epsilon \big]$, see \cite[Theorem 7]{Delyon2015}.
\end{proof}

As a corollary, we obtain the following McDiarmid's inequality of ``variance type".

\begin{corollary}\label{cor:md}
Let $(\Omega,\mathcal{F},\mathbb{P})$ be a probability space and let $(\Omega_i)_{i=1,...,n} $ be $n$ measurable subsets of $\Omega$. Let $X=(X_i)_{i=1,...,n}$ be $n$ independent random variables valued respectively in $(\Omega_i)_{i=1,...,n}$.
Consider a measurable function $f\colon \prod_{i=1}^{n} \Omega_i \rightarrow \mathbb{R}$ and real constants $m \in \R$ and $(v_i)_{i=1,...,n}$ such that
\begin{equation*}
\mathrm{Var} \big[ f(X_i, x_{-i}) \big] \leq v_{i}^2, \quad \text{a.s.}, \qquad
\big| f(X_i,x_{-i}) -\mathbb{E} \big[ f(X_i, x_{-i}) \big] \big| \leq m, \qquad \text{a.s.},
\end{equation*}
for all $i=1,...,n$ and for all $x_{-i} \in \Big( \prod_{j=1}^{i-1} \Omega_i \Big) \times \Big( \prod_{j=i+1}^{n} \Omega_j \Big)$.
Then, for any $\epsilon>0$,
\begin{equation}\label{eq:mc2}
\mathbb{P} \Big[ f(\mathbf{x}) -\mathbb{E} \big[ f(\mathbf{x}) \big] \geq \epsilon \Big]
\leq
\exp \Big( -\frac{\epsilon^{2}}{2\left(\sum_{i=1}^n v^2_i + \frac{m\epsilon}{3}\right)} \Big).
\end{equation}	
\end{corollary}

\begin{proof}
Define $Y_t = \mathbb{E}\left[  f(X) \mid X_1, \ldots, X_t \right] - \mathbb{E}\left[  f(X) \mid X_1, \ldots, X_{t-1} \right]$ and apply Proposition \ref{prop:exp}.
\end{proof}

\begin{lemma} \label{lm: convergence}
For all $k \in \mathbb{N}$, denote  {${\omega}_k= \frac{2}{k+2}$}.
Let $(u_k)_{k \in \mathbb{N}}$ and $(\gamma_k)_{k \in \mathbb{N}}$ be two sequences of real numbers. Assume that there exists a positive number $C$ such that
{\begin{equation} \label{eq:cond_cv_fw}
\gamma_{k+1} \leq (1- {\omega}_k) \gamma_k + C {\omega}_k^2 + u_k,
\end{equation}}
for all $k \in \mathbb{N}$.
{Then, for all $K \in \mathbb{N}^*$,
\begin{equation}\label{eq:fw}
\gamma_K \leq
\frac{4C}{K} + \sum_{k=0}^{K-1}\frac{(k+1)(k+2)}{K(K+1)} u_{k}.
\end{equation}}
\end{lemma}

\begin{proof}
We proof this lemma by induction on $K$.
We have ${\omega}_0=1$, thus taking $k=0$ in \eqref{eq:cond_cv_fw}, we obtain that $\gamma_1 \leq C  + u_0$, which proves the claim for $K=1$.
Let us assume that the claim holds true for some $K \in \mathbb{N}^*$. We deduce from \eqref{eq:cond_cv_fw} that
\begin{align*}
\gamma_{K+1} \leq \ & \Big(\frac{1}{K+2}+\frac{1}{(K+2)^2} \Big) 4C + \frac{K}{K+2} \Big( \sum_{k=0}^{K-1} \frac{(k+1)(k+2)}{K(K+1)} u_{k} \Big) + u_{K}\\
\leq \ & \frac{4C}{K+1} + \sum_{k=0}^{K} \frac{(k+1)(k+2)}{(K+1)(K+2)} u_{k}.
\end{align*}
Therefore the claim holds for $K+1$. This concludes the proof.
\end{proof}

\begin{lemma} \label{lem:new_bernoulli}
Let $A$, $B$, and $C$ be three random variables. Assume that $B$ is independent of $(A,C)$ and that $B \sim \Bern(\omega)$ for some $\omega \in [0,1]$. Let $F$ be a real-valued function of $(A,B,C)$. Assume that
$| F(A,1,C)- F(A,0,C) | \leq \delta, \text{a.s.}$
Finally, define $U= \mathbb{E} [F(A,B,C) \mid A,B] - \mathbb{E} [F(A,B,C) \mid A]$.
Then,
\begin{equation*}
\mathbb{E}[U \mid A]= 0, \qquad
U \leq \delta, \qquad
\mathbb{E} [U^2 \mid A] \leq \omega (1-\omega) \delta^2, \quad \text{a.s.}
\end{equation*}
\end{lemma}

\begin{proof}
The equality $\mathbb{E}[U \mid A]= 0$ is trivial.
We have $U= \mathbb{E}[Z \mid A,B]$, where
\begin{equation*}
Z= F(A,B,C)- \mathbb{E}[F(A,B,C) \mid A,C].
\end{equation*}
It is easy to verify that $Z \leq \delta$, a.s., which implies that $\mathbb{E}[U \mid A]= \mathbb{E}[Z \mid A] \leq \delta$. The first inequality is proved. For the second inequality, we first note that
\begin{equation*}
\mathbb{E}[Z^2 \mid A,C]
= \omega (1-\omega) (F(A,1,C)-F(A,0,C))^2,
\end{equation*}
as can be easily verified.
Thus $\mathbb{E}[Z \mid A] \leq \omega(1-\omega) \delta^2$.
Next by Jensen's inequality, we have $U^2 \leq \mathbb{E}[Z^2 \mid A,B]$. Therefore,
\begin{equation*}
\mathbb{E} [U^2 \mid A]
\leq \mathbb{E} \big[ \mathbb{E}[Z^2 \mid A,B] \mid A \big]= \mathbb{E}[Z^2 \mid A] \leq \omega(1-\omega) \delta^2,
\end{equation*}
as was to be proved.
\end{proof}

The following lemma is an elementary property of the conditional expectation. For the sake of simplicity, we only state it (and prove it) with discrete random variables.

\begin{lemma} \label{lemma_stupid}
Let $X$, $Y$, and $Z$ be three random variables. Assume that $Y$ and $Z$ are discrete and that $Z$ is independent of $(X,Y)$. Then, $\mathbb{E}\big[ X \mid Y,Z \big]= \mathbb{E} \big[ X \mid Y \big]$.
\end{lemma}
\begin{proof}
By definition, $\mathbb{E}\big[ X \mid Y,Z \big]= \phi(Y,Z)$, where $\phi$ is defined as follows: for any pair $(y,z)$ such that $\mathbb{P}\big[ Y=y \text{ and } Z= z\big]  \neq 0$,
\begin{equation*}
\phi(y,z)
=
\frac{\mathbb{E}\big[ X \mathbf{1}_{Y=y} \mathbf{1}_{Z=z} \big]}{\mathbb{P} \big[ Y=y \text{ and } Z=z \big]}
=
\frac{\mathbb{E}\big[ X \mathbf{1}_{Y=y} \big]}{\mathbb{P} \big[ Y=y \big]},
\end{equation*}
since $Z$ is independent of $(X,Y)$. Thus $\phi$ does not depend on $Z$ and the result follows.
\end{proof}

\bibliographystyle{plain}
\bibliography{biblio.bib}

\end{document}